\documentclass[12pt]{article}
\usepackage[margin=1in]{geometry}

\usepackage{amsmath}
\usepackage{amssymb}
\usepackage{microtype}
\usepackage{amsthm}
\usepackage{tikz}

\newtheorem{thm}{Theorem}[section]
\newtheorem{defn}[thm]{Definition}
\newtheorem{lemma}[thm]{Lemma}

\newtheorem{cor}[thm]{Corollary}

\newtheorem{conj}[thm]{Conjecture}

\newcommand{\torusMd}[0]{{\mathbb Z}^d_m}
\newcommand{\torustwod}[0]{Q_d}
\newcommand{\sumrange}[0]{(A_0, \ldots, A_{m-1})}
\newcommand{\altAB}[0]{{\rm alt}(A,B)}
\newcommand{\rst}[1]{\ensuremath{{\mathbin\upharpoonright}%
\raise-.5ex\hbox{$#1$}}}

\begin{document}

\title{$H$-coloring tori}

\author{John Engbers
\and
David Galvin}

\date{\today\thanks{$\{$jengbers, dgalvin1$\}$@nd.edu; Department of Mathematics,
University of Notre Dame, Notre Dame IN 46556. Galvin in part supported by National Security Agency grant H98230-10-1-0364.}}

\maketitle

\begin{abstract}
For graphs $G$ and $H$, an $H$-coloring of $G$ is a function from the vertices of $G$ to the vertices of $H$ that preserves adjacency. $H$-colorings encode graph theory notions such as independent sets and proper colorings, and are a natural setting for the study of hard-constraint models in statistical physics.

We study the set of $H$-colorings of the even discrete torus $\torusMd$, the graph on vertex set $\{0, \ldots, m-1\}^d$ ($m$ even) with two strings adjacent if they differ by $1$ (mod $m$) on one coordinate and agree on all others. This is a bipartite graph, with bipartition classes ${\mathcal E}$ and ${\mathcal O}$. In the case $m=2$ the even discrete torus is the discrete hypercube or Hamming cube $\torustwod$, the usual nearest neighbor graph on $\{0,1\}^d$.

We obtain, for any $H$ and fixed $m$, a structural characterization of the space of $H$-colorings of $\torusMd$. We show that it may be partitioned into an exceptional subset of negligible size (as $d$ grows) and a collection of subsets indexed by certain pairs $(A,B) \in V(H)^2$, with each $H$-coloring in the subset indexed by $(A,B)$ having all but a vanishing proportion of vertices from ${\mathcal E}$ mapped to vertices from $A$, and all but a vanishing proportion of vertices from ${\mathcal O}$ mapped to vertices from $B$. This implies a long-range correlation phenomenon for uniformly  chosen $H$-colorings of $\torusMd$ with $m$ fixed and $d$ growing.

The special pairs $(A,B) \in V(H)^2$ are characterized by every vertex in $A$ being adjacent to every vertex in $B$, and having $|A||B|$ maximal subject to this condition. Our main technical result is an upper bound on the probability, for an arbitrary edge $uv$ of $\torusMd$, that in a uniformly  chosen $H$-coloring $f$ of $\torusMd$ the pair $(\{f(w):w\in N_u\},\{f(z):z\in N_v\})$ is not one of these special pairs (where $N_\cdot$ indicates neighborhood).

Our proof proceeds through an analysis of the entropy of $f$, and extends an approach of Kahn, who had considered the case of $m=2$ and $H$ a doubly infinite path. All our results generalize to a natural weighted model of $H$-colorings.
\end{abstract}

\section{Introduction and statement of results} \label{sec-intro}

For $G=(V(G),E(G))$ a simple, loopless graph, and $H=(V(H),E(H))$ a graph without multiple edges but perhaps with loops, an {\em $H$-coloring} of $G$, or {\em homomorphism} from $G$ to $H$, is a function $f:V(G) \rightarrow V(H)$ that preserves adjacency, that is, which satisfies $f(u)f(v) \in E(H)$ whenever $uv \in E(G)$. We write ${\rm Hom}(G,H)$ for the set of $H$-colorings of $G$. (Unless explicitly stated otherwise, all graphs in this paper will be finite. For graph theory background, see e.g. \cite{Bollobas2}, \cite{Diestel}.)

$H$-colorings provide a unifying framework for a number of important graph theory notions. For example, the set ${\rm Hom}(G,K_q)$ (where $K_q$ is the complete loopless graph on $q$ vertices) coincides with the set of proper $q$-colorings of $G$, and the set ${\rm Hom}(G,H_{\rm ind})$ (where $H_{\rm ind}$ consists of two vertices joined by an edge, with a loop at one of the vertices) may be identified with the set of independent sets of $G$, via the preimage of the unlooped vertex.

$H$-colorings also have a natural statistical physics interpretation as configurations in {\em hard-constraint spin models}. Here, the vertices of $G$ are thought of as sites that are occupied by particles, with edges of $G$ representing pairs of bonded sites. The vertices of $H$ are the different types of particles (or spins), and the occupation rule is that bonded sites must be occupied by pairs of particles that are adjacent in $H$. A legal configuration in such a spin model is exactly an $H$-coloring of $G$. The case of proper $q$-colorings corresponds to the {\em zero-temperature $q$-state anti-ferromagnetic Potts model}, while the case of independent sets corresponds to the {\em hard-core lattice gas model}. (See for example \cite{BrightwellWinkler2} for a discussion of these models from a combinatorial point of view, and \cite{Sokal} for a statistical physics oriented discussion.) Another important hard-constraint model is the {\em Widom-Rowlinson model} (or {\em WR model}), introduced in \cite{WidomRowlinson} as a model of liquid-vapor phase transitions. Here $H_{\rm WR}$ is the completely looped path on $3$ vertices.

There have been numerous papers devoted to the study of the space of $H$-colorings of particular graphs and families of graphs, for various special instances of $H$. Some recent papers (see for example \cite{BorgsChayesDyerTetali}, \cite{BrightwellWinkler}, \cite{EngbersGalvin}, \cite{Galvin-spin} and \cite{GalvinTetali-weighted}) have taken a broader approach, treating the space of $H$-colorings for arbitrary $H$. The present paper falls into this category.

\medskip

Many of the graphs $G$ on which it is natural (from a statistical physics viewpoint) to study ${\rm Hom}(G,H)$ are regular (all vertices have the same degree) and bipartite (the vertex set splits into two classes with all edges going between classes). Examples include the hypercubic lattice ${\mathbb Z}^d$, the hexagonal lattice and the Bethe lattice (regular tree). For this reason much attention has been focused on this special case, and that is also where our focus lies.

In \cite{GalvinTetali-weighted}, an entropy approach was taken to obtain nearly matching upper and lower bounds on $|{\rm Hom}(G,H)|$ for arbitrary $H$ and $d$-regular bipartite $G$, specifically
\begin{equation} \label{from-GT}
\eta(H)^\frac{|V(G)|}{2}  \leq  |{\rm Hom}(G,H)| \leq \eta(H)^\frac{|V(G)|}{2} 2^\frac{|V(G)|}{2d},
\end{equation}
with $\eta(H)$ a certain parameter that will be defined presently. In \cite{EngbersGalvin}, this work was extended considerably. For all $H$ and $k \in V(H)$, optimal numbers $a^+(k)$ and $a^-(k)$ are constructed with the following property: for each $\varepsilon > 0$, if $f$ is uniformly chosen from ${\rm Hom}(G,H)$, then (for suitably large $d$) with high probability the proportion of vertices of $G$ mapped to $k$ is between $a^-(k)-\varepsilon$ and $a^+(k)+\varepsilon$.

Let $G$ be a bipartite graph with fixed bipartition ${\mathcal E} \cup {\mathcal O}$. For $A, B \subseteq V(H)$ with all vertices of $A$ adjacent to all vertices of $B$, a {\em pure-$(A,B)$ coloring} is an $f \in {\rm Hom}(G,H)$ with $f(u) \in A$ for all $u \in {\mathcal E}$ and $f(v) \in B$ for all $v \in {\mathcal O}$. If $G$ is regular and has $n$ vertices, then the number of pure-$(A,B)$ colorings of $G$ is $\left(|A||B|\right)^{n/2}$. An intuition driving the results of \cite{EngbersGalvin} and \cite{GalvinTetali-weighted} is that in a certain sense, most $f \in {\rm Hom}(G,H)$ are close to pure-$(A,B)$ colorings for some $(A,B)$ that maximizes $|A||B|$ (the maximum value is the $\eta(H)$ of (\ref{from-GT}); note that there may be many $(A,B)$ that achieve the maximum).

Such an intuition cannot be formalized for {\em all} regular bipartite $G$ --- for example, by the independence of the coloring on different components of a disconnected graph, it is easy to see that the intuition cannot be true for a graph that consist of a large number of small components. If, however, we are working with connected graphs with reasonable expansion (meaning that each subset of vertices from one partition class has a reasonably large number of neighbors in the other class) then we might expect it to be true that most $f \in {\rm Hom}(G,H)$ are close to pure-$(A,B)$ colorings for some $(A,B)$. This is shown for random regular bipartite graphs, for example, in \cite{EngbersGalvin}, and the proof critically uses the excellent expansion of random graphs.

For other graphs with weaker but still good expansion we expect similar results. One family of graphs that is of particular interest, given the statistical physics interpretation of $H$-colorings, is the integer lattice ${\mathbb Z}^d$ with the usual nearest neighbor adjacency, together with its finite analog the discrete torus $\torusMd$, the graph obtained from an axis-parallel box in ${\mathbb Z}^d$ by identifying opposite faces. These graphs have been the focus of study for particular homomorphism models (see e.g. \cite{GalvinKahn} for independent sets and \cite{BorgsChayesFriezeKimTetaliVigodaVu} for proper colorings), as well as for general $H$-colorings (see e.g. \cite{BorgsChayesDyerTetali}).

Formally, for each $d \geq 1$ and even $m \geq 2$, the even discrete torus $\torusMd$ is the graph on vertex set $V=\{0,1,\ldots,m-1\}^d$ with edge set $E$ consisting of all pairs of strings that differ by exactly 1 (mod $m$) on exactly one coordinate. For $m \geq 4$ it is $2d$-regular and bipartite while for $m=2$ it is $d$-regular and bipartite. We denote by ${\mathcal E}$ the bipartition class of vertices the sum of whose coordinates is even, and by ${\mathcal O}$ the complementary class.
In the case $m=2$, the even discrete torus is isomorphic to the familiar Hamming cube or discrete hypercube (the graph on vertex set $\{0,1\}^d$ with edge set consisting of all pairs of strings that differ on exactly one coordinate). For this special case we use the more familiar notation $\torustwod$.

In \cite{EngbersGalvin} information is given about the number of occurrences of each color in a uniformly chosen $H$-coloring of $\torusMd$, but no information is given about how the vertices of a particular color are distributed between ${\mathcal E}$ and ${\mathcal O}$. Some special cases of this problem have been previously addressed, as we now discuss. (Note that we frequently refer to elements of $V(H)$ as {\em colors}, and say that a vertex of $\torusMd$ is {\em colored} $k$ if its image in the $H$-coloring under consideration is $k$.)

In \cite{KorshunovSapozhenko}, in the course of deriving the asymptotic formula
\begin{equation} \label{KS-result}
|{\rm Hom}(Q_d,H_{\rm ind})|=(2\sqrt{e}+o(1))2^{2^{d-1}}
\end{equation}
(as $d \rightarrow \infty$),
Korshunov and Sapozhenko showed that if $I$ is a uniformly chosen independent set from $Q_d$ (that is, if $I$ is the preimage of the unlooped vertex in a uniformly chosen $f$ from ${\rm Hom}(Q_d,H_{\rm ind})$), then with high probability $I$ has size close to $2^d/4$ and is contained almost entirely in a single partition class. Kahn \cite{Kahn} and Galvin \cite{Galvin-Qdthresh} extended these results to the case of $I$ chosen from the set of independent sets according to the hard-core distribution with parameter $\lambda$, that is, the distribution in which each set $I$ is chosen with probability proportional to $\lambda^{|I|}$ for some $\lambda > 0$ (Korshunov and Sapozhenko's setting is $\lambda = 1$).

In \cite{Kahn2}, Kahn considered the set ${\rm Hom}(Q_d,{\mathbb Z})/\!\!\sim$ (where ${\mathbb Z}$ is given a graph structure by declaring consecutive integers to be adjacent, and $\sim$ is the equivalence relation defined by $h\sim g$ if and only if $h-g$ is a constant function). Answering a question of Benjamini, H\"aggstr\"om and Mossel \cite{BenjaminiHaggstromMossel}, he showed that if $f$ is a uniformly chosen element from this set (a ``cube-indexed random walk''), then with high probability $f$ takes on only constantly many values (independent of $d$). Extending this work, Galvin \cite{Galvin-HomstoZ} showed that in fact $f$ takes on only at most five (consecutive) values, that $f$ is constant on all but $o(2^d)$ (actually, at most $g(d)$ for any $g(d)=\omega(1)$) vertices on one of the two bipartition classes of $Q_d$, and that on the other partition classes each of two values appear on $(1/4-o(1))2^d$ of the vertices.
Using a correspondence between ${\rm Hom}(Q_d,{\mathbb Z})/\!\!\sim$ and ${\rm Hom}(Q_d,K_3)$, the results of \cite{Galvin-HomstoZ} also answer the question of the structure of a typical (uniformly chosen) proper $3$-coloring of $Q_d$. In the process of showing
\begin{equation} \label{G-count-of-3-cols}
|{\rm Hom}(Q_d,K_3)|=(6e+o(1))2^{2^{d-1}}
\end{equation}
it is shown in \cite{Galvin-HomstoZ} that ${\rm Hom}(Q_d,K_3)$ may be partitioned into an exceptional subset of size $o(1)|{\rm Hom}(Q_d,K_3)|$, and six equal sized subsets, with the property that within each of these six subsets, all colorings are constant on all but $o(2^d)$ (again, actually at most $g(d)$ for any $g(d)=\omega(1)$) vertices on one of the two bipartition classes of $Q_d$, and on the other partition classes each of two colors appear on $(1/4-o(1))2^d$ of the vertices. Peled \cite{Peled} has recently extended these results on the $3$-coloring and cube-indexed random walk models to more general tori.

One of the main purposes of this paper is to extend these structural characterizations of ${\rm Hom}(Q_d,H_{\rm ind})$ and ${\rm Hom}(Q_d,K_3)$ to arbitrary $H$ and from $Q_d$ to $\torusMd$ for all even $m$. We also extend to a general class of probability distributions on ${\rm Hom}(\torusMd,H)$ that are very natural to consider from a statistical physics standpoint. Fix a set of positive weights $\Lambda = \{\lambda_i:i \in V(H)\}$ indexed by the vertices of $H$. We think of the magnitude of $\lambda_k$ as measuring how likely color $k$ is to appear at each vertex. This can be formalized by giving each $f \in {\rm Hom}(\torusMd,H)$ weight $w_\Lambda(f) = \prod_{v \in V(\torusMd)} \lambda_{f(v)}$ and probability
$$
p_\Lambda(f) = \frac{w_\Lambda(f)}{Z_\Lambda(\torusMd,H)}
$$
where $Z_\Lambda(\torusMd,H)=\sum_{f \in {\rm Hom}(\torusMd,H)} w_\Lambda(f)$ is the appropriate normalizing constant or {\em partition function}. When all weights are $1$, $Z_\Lambda(\torusMd,H)$ is simply counting the number of $H$-colorings, and $p_\Lambda$ is uniform measure. (For a good introduction to these distributions see for example \cite{BrightwellWinkler2}.) Because of a technical limitation of one step in our proof, all $\lambda_i$'s under consideration in this paper will be rational.

Throughout the paper, we use the standard Landau notation, with $f=o(g)$ and $f=\omega(g)$ indicating, respectively, that $f/g \rightarrow 0$ and $f/g \rightarrow \infty$ as $d \rightarrow \infty$; $f=O(g)$ and $f=\Omega(g)$ indicating, respectively, that $|f| < C|g|$ and $|f| > C|g|$ for some constant $C$; and $f=\Theta(g)$ indicating that both $f=O(g)$ and $f=\Omega(g)$ hold. We will always think of $d$ as the variable in our functions, with $m$, $H$ and (when present) $\Lambda$ some fixed parameters, and so all implicit constants depend only on $m$, $H$ and $\Lambda$, but not on $d$. Where necessary we will always assume that $d$ is large enough to support our assertions. For $S \subseteq {\rm Hom}(\torusMd,H)$ and $T \subseteq V(H)$ we write $w_\Lambda(S)$ for $\sum_{f \in S} w_\Lambda(f)$ and $\lambda_T$ for $\sum_{k \in T} \lambda_k$. With $A \sim B$ indicating that every vertex in $A$ is adjacent to every vertex in $B$, set
$$
\eta_\Lambda(H) = \max \left\{\lambda_A \lambda_B \colon A, B \subseteq V(H), ~\!A \sim B \right\}
$$
and
$$
{\mathcal M}_\Lambda(H) = \left\{(A,B) \in V(H)^2 \colon A \sim B, ~\! \lambda_A \lambda_B = \eta_\Lambda(H)\right\}.
$$
We denote by $N_x$ the set of neighbors of $x$, and later use $N(X)$ for $\cup_{x\in X} N_x$.

We now state our first main result, a structural decomposition of ${\rm Hom}(\torusMd,H)$ (in the presence of  weight-set $\Lambda$) into finitely many classes of similar-looking colorings.
\begin{thm} \label{thm-structure-of-hom-space-coarse}
Fix $H$, rational $\Lambda$ and $m \geq 2$ even. There is a partition of ${\rm Hom}(\torusMd,H)$ into $|{\mathcal M}_\Lambda(H)|+1$ classes as
$$
{\rm Hom}(\torusMd,H) = D_\Lambda(0) \cup \cup_{(A,B) \in {\mathcal M}_\Lambda(H)} D_\Lambda(A,B)
$$
with the following properties.
\begin{enumerate}
\item \label{main-thm-coarse-item-1} $w_\Lambda(D_\Lambda(0)) \leq 2^{-\Omega(d)} Z_\Lambda(\torusMd,H)$.
\item \label{main-thm-coarse-item-2} For each $(A,B) \in {\mathcal M}_\Lambda(H)$ and $f \in D_\Lambda(A,B)$, the number of vertices $v \in {\mathcal E}$ (resp. ${\mathcal O}$) with $f(v) \not \in A$ (resp. $f(v) \not \in B$) is at most $(m-\Omega(1))^d$, and moreover all but at most $(m-\Omega(1))^d$ vertices $w$ of ${\mathcal O}$ (resp. ${\mathcal E}$) have the property that all colors from $A$ (resp. $B$) appear on $N_w$.
\end{enumerate}
\end{thm}
This decomposition already gives us significant information about the structure of ${\rm Hom}(\torusMd,H)$ and the distribution $p_\Lambda$ on ${\rm Hom}(\torusMd,H)$. For the purpose of obtaining long-range influence results (see Section \ref{sec-influence}), we need a slightly stronger decomposition result that in addition quantifies the number of vertices of each color in an arbitrary element of each partition class as well as the sizes of the partition classes. In what follows we use $X=Y(1 \pm 2^{-\Omega(d)})$ to indicate $|X/Y - 1| \leq 2^{-\Omega(d)}$.
\begin{thm} \label{thm-structure-of-hom-space}
Fix $H$, rational $\Lambda$ and $m \geq 2$ even. There is a partition of ${\rm Hom}(\torusMd,H)$ into $|{\mathcal M}_\Lambda(H)|+1$ classes as
$$
{\rm Hom}(\torusMd,H) = C_\Lambda(0) \cup \cup_{(A,B) \in {\mathcal M}_\Lambda(H)} C_\Lambda(A,B)
$$
with the following properties.
\begin{enumerate}
\item \label{main-thm-item-1} $w_\Lambda(C_\Lambda(0)) \leq 2^{-\Omega(d)} Z_\Lambda(\torusMd,H)$.
\item \label{main-thm-item-2} For each $(A,B) \in {\mathcal M}_\Lambda(H)$, $f \in C_\Lambda(A,B)$, $k \in A$ and $\ell \in B$, the proportion of vertices of ${\mathcal E}$ (resp. ${\mathcal O}$) colored $k$ (resp. $\ell$) is within $2^{-\Omega(d)}$ of $\lambda_k/\lambda_A$ (resp. $\lambda_\ell/\lambda_B$).
\item \label{main-thm-item-3} If $A \neq B$ is such that $(A,B), (B,A) \in {\mathcal M}_\Lambda(H)$ then
$$
w_\Lambda(C_\Lambda(A,B)) = w_\Lambda(C_\Lambda(B,A))\left(1 \pm 2^{-\Omega(d)}\right).
$$
\item \label{main-thm-item-4} If $(A,B), (\tilde{A},\tilde{B}) \in {\mathcal M}_\Lambda(H)$ are such that $\varphi(A)=\tilde{A}$ and $\varphi(B)=\tilde{B}$ for some weight preserving automorphism $\varphi$ of $H$, then
$$
w_\Lambda(C_\Lambda(A,B))=w_\Lambda(C_\Lambda(\tilde{A},\tilde{B}))\left(1 \pm 2^{-\Omega(d)}\right).
$$
\item \label{main-thm-item-5} For each $(A,B) \in {\mathcal M}_\Lambda(H)$, $x \in {\mathcal E}$, $y \in {\mathcal O}$, $k \in A$ and $\ell \in B$,
$$
p_\Lambda(f(x)=k|f \in C_\Lambda(A,B)) = \frac{(1 \pm 2^{-\Omega(d)})\lambda_k}{\lambda_A}
$$
and
$$
p_\Lambda(f(y)=\ell|f \in C_\Lambda(A,B)) = \frac{(1 \pm 2^{-\Omega(d)})\lambda_\ell}{\lambda_B}.
$$
\end{enumerate}
\end{thm}

Theorem \ref{thm-structure-of-hom-space} does not make a general statement about the relative sizes of the $C_\Lambda(A,B)$'s, but there are two important situations in which we do obtain some information. It will helpful at this point to define the notion of an {\em approximate equipartition}.
\begin{defn}
Fix $H$, rational $\Lambda$ and $m \geq 2$ even. An {\em approximate equipartition} of ${\rm Hom}(\torusMd,H)$ is a partition into $|{\mathcal M}_\Lambda(H)|+1$ classes satisfying conditions \ref{main-thm-item-1}, \ref{main-thm-item-2} and \ref{main-thm-item-5} of Theorem \ref{thm-structure-of-hom-space}, as well as the condition that for all $(A,B), (A',B') \in {\mathcal M}_\Lambda(H)$ we have
$$
w_\Lambda(C_\Lambda(A,B)) = \left(1\pm 2^{-\Omega(d)}\right)w_\Lambda(C_\Lambda(A',B')).
$$
\end{defn}

A corollary of statements \ref{main-thm-item-1} and \ref{main-thm-item-3} is that if ${\mathcal M}_\Lambda(H) = \{(A,B),(B,A)\}$ for some $A \neq B$ (as, for example, in the case $H=H_{\rm ind}$ for arbitrary $\Lambda$), then the partition of ${\rm Hom}(\torusMd,H)$ from Theorem \ref{thm-structure-of-hom-space} is an approximate equipartition with
$$
w_\Lambda\left(C_\Lambda(A,B)\right),~\! w_\Lambda\left(C_\Lambda(B,A)\right) = Z_\Lambda(\torusMd,H)\left(\frac{1}{2}\pm 2^{-\Omega(d)}\right).
$$
Also, if ${\mathcal M}_\Lambda(H) = \{(A,A)\}$ for some $A$ then the partition of ${\rm Hom}(\torusMd,H)$ from Theorem \ref{thm-structure-of-hom-space} is trivially an approximate equipartition with $w_\Lambda\left(C_\Lambda(A,A)\right) = Z_\Lambda(\torusMd,H)\left(1 - 2^{-\Omega(d)}\right)$. These are in a sense the two generic situations, as for every $H$, if the weights $\lambda_i$ are chosen from any continuous distribution supported on $\{x \in {\mathbb R}^{|V(H)|}: x > 0\}$ then with probability $1$ one of these two situations will occur.

A corollary of statements \ref{main-thm-item-1} and \ref{main-thm-item-4} is that if ${\mathcal M}_\Lambda(H)$ is transitive, that is, if for each $(A,B), (\tilde{A},\tilde{B}) \in {\mathcal M}_\Lambda(H)$ there is a weight preserving automorphism $\varphi$ of $H$ with $\varphi(A)=\tilde{A}$ and $\varphi(B)=\tilde{B}$, then the partition of ${\rm Hom}(\torusMd,H)$ is an approximate equipartition with
$$
w_\Lambda\left(C_\Lambda(A,B)\right) = Z_\Lambda(\torusMd,H)\left(\frac{1}{|{\mathcal M}_\Lambda(H)|}\pm 2^{-\Omega(d)}\right).
$$
This is far from a generic situation, but is the case for a number of very important examples, such as the uniform proper $q$-coloring model ($H=K_q$ and $\Lambda=(1, \ldots, 1)$), where it easily seen that
$$
|{\mathcal M}_\Lambda(K_q)| = \left\{ \begin{array}{ll}
{q \choose q/2} & \mbox{if $q$ even} \\
{q \choose (q-1)/2} + {q \choose (q+1)/2} & \mbox{if $q$ odd},
\end{array} \right.
$$
or more concisely $|{\mathcal M}_\Lambda(K_q)| = (1+{\mathbf 1}_{\{q~{\rm odd}\}}){q \choose [q/2]}$.
(Note that ${\mathcal M}(K_q)$ consists of all pairs $(A,B)$ with $A$ and $B$ disjoint, $A \cup B = V(K_q)$, and $|A|$, $|B|$ as near equal as possible). Another example of this behavior is the uniform Widom-Rowlinson model ($H$ the complete looped path on three vertices, or equivalently the complete looped graph on $\{1,2,3\}$ with edge $13$ removed). In this case we have ${\mathcal M}_\Lambda(H)=\{(A,A),(B,B)\}$ with $A=\{1,2\}$ and $B=\{2,3\}$.

The existence of these equipartitions is what drives our long-range influence results Corollaries \ref{cor-ind-inf}, \ref{cor-q-col-inf} and \ref{cor-WR-inf} in Section \ref{sec-influence}. A representative result from that section is the following: in a proper $q$-coloring of $Q_d$ chosen uniformly conditioned on a particular vertex $v \in {\mathcal E}$ being colored $1$, the probability that another vertex $u \in {\mathcal E}$ is colored $1$ is close to $2/q$, whereas the probability that a vertex $w \in {\mathcal O}$ is colored $1$ is close to $0$, regardless of the distances between $u$, $v$ and $w$.

In general, we cannot say anything more about the relative ($\Lambda$-weighted) sizes of the $C_\Lambda(A,B)$, and indeed we can construct examples to show that various different types of behaviors can occur. We postpone a discussion of this, together with a conjecture concerning the sizes, to Section \ref{sec-openprobs}.

\medskip

The proof of Theorem \ref{thm-structure-of-hom-space} is based on the notion of an ideal edge. Let $H$ and $f \in {\rm Hom}(\torusMd,H)$ be given. Say that an edge $e=uv \in E$ (with $u \in {\mathcal E}$) is {\em ideal} (with respect to $f$) if $f(N_u) = B$ and $f(N_v) = A$ for some $(A,B) \in {\mathcal M}(H)$. We will only be interested in the probability that a particular edge is not ideal with respect to $f$, when $f$ is chosen uniformly from ${\rm Hom}(\torusMd,H)$. Note that by the symmetry of the torus, this probability is independent of the particular edge we choose.
Our main technical result is the following.
\begin{thm} \label{thm-small-prob-edge-not-ideal}
Fix $H$, $m \geq 2$ even, and $e \in E$. If $f$ is chosen uniformly from ${\rm Hom}(\torusMd,H)$ then
$$
\Pr(\mbox{$e$ is not ideal with respect to $f$}) \leq 2^{-\Omega(d)}.
$$
\end{thm}
The analogous result for $m=2$ and $H={\mathbb Z}$ (with two elements of ${\rm Hom}(\torustwod,{\mathbb Z})$ identified if they differ by a constant) was proved by Kahn in \cite{Kahn2}, and our proof follows similar lines. A standard trick of comparing a weighted $H$-coloring model to a uniform $H'$-coloring model for a certain graph $H'$ (depending on $H$ and $\Lambda$) makes the generalization from uniform to arbitrary $\Lambda$ relatively straightforward.

\medskip

The paper is laid out as follows. In Section \ref{sec-influence} we discuss a long-range influence phenomenon that is implied by Theorem \ref{thm-structure-of-hom-space}. In Section \ref{sec-cors} we derive Theorem \ref{thm-structure-of-hom-space} from Theorem \ref{thm-small-prob-edge-not-ideal}. We then give the proof of Theorem \ref{thm-small-prob-edge-not-ideal} in Section \ref{sec-tech-proof}. In Section \ref{sec-varyingwithd} we discuss the extent to which our proof goes through for the proper $q$-coloring model when $q$ is allowed to grow with $d$. Some open problems and conjectures are discussed in Section \ref{sec-openprobs}.

\section{Long-range influence} \label{sec-influence}

Roughly speaking we say that a distribution $p_\Lambda$ on ${\rm Hom}(\torusMd,H)$ exhibits {\em long-range influence} if the distribution of $p_\Lambda$ restricted to a single vertex $x$ is sensitive to conditioning on the color of another vertex $y$, even in the limit as $d$ and the distance from $x$ to $y$ go to infinity.

More formally, given a graph $H$, a weight set $\Lambda$ and even $m$, we say that the $\Lambda$-weighted $H$-coloring model on $\torusMd$ exhibits long-range influence if there is a choice of $x, y \in V$ and $k, \ell \in V(H)$ (actually a sequence of choices, one for each $d$) with ${\rm dist}(x,y)=\omega(1)$ (where ${\rm dist}$ is usual graph distance) such that
\begin{equation} \label{def-lri}
\frac{p_\Lambda(f(x)=k|f(y)=\ell)}{p_\Lambda(f(x)=k)} \not \rightarrow 1~\mbox{as $d \rightarrow \infty$}.
\end{equation}
Theorem \ref{thm-structure-of-hom-space} strongly implies such a phenomenon, at least in the case where the partition of ${\rm Hom}(\torusMd,H)$ guaranteed by Theorem \ref{thm-structure-of-hom-space} is an approximate equipartition. The following is an immediate corollary of Theorem \ref{thm-structure-of-hom-space}, and in particular statement \ref{main-thm-item-5} of that theorem.
\begin{thm} \label{thm-long-range-inf}
Fix $H$, rational $\Lambda$ and $m \geq 2$ even. Suppose that the partition of ${\rm Hom}(\torusMd,H)$ from Theorem \ref{thm-structure-of-hom-space} is an approximate equipartition. Fix $k, \ell \in V(H)$. For all $x \in {\mathcal E}$ we have
$$
p_\Lambda(f(x)=k) = \left(\frac{1}{|{\mathcal M}_\Lambda(H)|} \pm 2^{-\Omega(d)}\right)\sum_{(A,B) \in {\mathcal M}_\Lambda(H) \colon k \in A} \frac{\lambda_k}{\lambda_A}
$$
(and by symmetry this is also true for $x \in {\mathcal O}$).
On the other hand, if $x, y \in {\mathcal E}$ then
$$
p_\Lambda(f(x)=k | f(y)=\ell) = \left(\frac{1}{|{\mathcal M}_\Lambda(H)|}\pm 2^{-\Omega(d)}\right)\sum_{(A,B) \in {\mathcal M}_\Lambda(H) \colon \ell,\, k \in A} \frac{\lambda_k}{\lambda_A}
$$
and if $x \in {\mathcal E}$ and $y \in {\mathcal O}$ then
$$
p_\Lambda(f(x)=k | f(y)=\ell) = \left(\frac{1}{|{\mathcal M}_\Lambda(H)|} \pm 2^{-\Omega(d)}\right)\sum_{(A,B) \in {\mathcal M}_\Lambda(H) \colon k \in A,\, \ell \in B} \frac{\lambda_k}{\lambda_A}.
$$
\end{thm}
By choosing $k, \ell$ appropriately, these three quantities can be made to be different (in the limit as $d \rightarrow \infty$). Rather than stating an unwieldy general proposition to this effect, we illustrate it with three examples. It will be helpful first to set up some notation. Fix $m$, $H$ and $\Lambda$. For each $d \in {\mathbb N}$ and $x \in V$, we define the {\em occupation probability vector} $\vec{v}_d(x)$ by
$$
\vec{v}_d(x) = \left(p_\Lambda(f(x)=k) \colon k \in V(H) \right).
$$
(We suppress dependance on $m$, $H$ and $\Lambda$ to aid readability.) If the choice of $f$ is conditioned on an event $E$ we use $\vec{v}_d(x|E)$ to denote the conditional occupation probability vector, that is,
$$
\vec{v}_d(x|E) = \left(p_\Lambda(f(x)=k|E) \colon k \in V(H)\right).
$$
In what follows we use $d_\infty(\cdot, \cdot)$ for $\ell_\infty$ distance.

Our first example is the independent set model, that is, $H=H_{\rm ind}$ where
$V(H_{\rm ind})=\{v_{\rm in}, v_{\rm out}\}$ and $E(H_{\rm ind})=\{v_{\rm in}v_{\rm out},v_{\rm out}v_{\rm out}\}$. We list $v_{\rm in}$ first in the occupation and conditional occupation probability vectors. Our weighting vector will assign rational weight $\lambda$ to $v_{\rm in}$ and weight $1$ to $v_{\rm out}$. (This is the {\em hard-core model} with fugacity $\lambda$, results on which from \cite{Kahn} have been discussed earlier.) Noting that ${\mathcal M}_\lambda(H_{\rm ind}) = \{(A,B),(B,A)\}$ where $A=\{v_{\rm in}, v_{\rm out}\}$ and $B=\{v_{\rm out}\}$, we have the following.
\begin{cor} \label{cor-ind-inf}
Fix $m \geq 2$ even and rational $\lambda > 0$. For all $x \in V$ we have
$$
d_\infty\left(\vec{v}_d(x), \left(\frac{\lambda}{2(1+\lambda)},\frac{2+\lambda}{2(1+\lambda)}\right)\right) \leq 2^{-\Omega(d)}.
$$
On the other hand, if $x, y \in {\mathcal E}$ then
$$
d_\infty\left(\vec{v}_d(x|\{f(y)=v_{\rm in}\}),  \left(\frac{\lambda}{1+\lambda},\frac{1}{1+\lambda}\right)\right) \leq 2^{-\Omega(d)}
$$
and if $x \in {\mathcal E}$ and $y \in {\mathcal O}$ then
$$
d_\infty\left(\vec{v}_d(x|\{f(y)=v_{\rm in}\}),   \left(0,1\right)\right) \leq 2^{-\Omega(d)}.
$$
\end{cor}
(This result was earlier proven in \cite{Galvin-Qdthresh} for $m=2$ and all $\lambda$ (not necessarily rational) satisfying $\lambda > cd^{-1/3}\log d$ for some constant $c>0$.)

Our second example is the uniform proper $q$-coloring model ($H=K_q$ where
$V(K_q)=\{1, \ldots, q\}$ and $E(K_q)=\{ij:i \neq j\}$, and $\Lambda={\vec 1}$). We list color $1$ first in the occupation and conditional occupation probability vectors. By our earlier observation that ${\mathcal M}(H)$ consists of all pairs $(A,B)$ with $A \cup B = \{1, \ldots, q\}$, $A \cap B = \emptyset$ and $|A|-|B| \in \{0, \pm 1\}$, we get the following via a routine calculation.
\begin{cor} \label{cor-q-col-inf}
Fix $m \geq 2$ even and $q \in {\mathbb N}$. For all $x \in V$ we have
$$
\vec{v}_d(x) = \left(\frac{1}{q},\ldots,\frac{1}{q}\right).
$$
On the other hand, if $x, y \in {\mathcal E}$ then
$$
d_\infty\left(\vec{v}_d(x|f(y)=1), \left(\frac{2}{q},\frac{q-2}{q(q-1)}, \ldots, \frac{q-2}{q(q-1)} \right)\right) \leq 2^{-\Omega(d)}
$$
and if $x \in {\mathcal E}$ and $y \in {\mathcal O}$ then
$$
d_\infty\left(\vec{v}_d(x|f(y)=1), \left(0,\frac{1}{q-1}, \ldots, \frac{1}{q-1}\right)\right) \leq 2^{-\Omega(d)}.
$$
\end{cor}
The exact equality for $\vec{v}_d(x)$ here follows by symmetry. This corollary, in the special case $m=2$ and $q=3$, was proved in \cite{Galvin-3colQdmix} (and is implicit in \cite{Galvin-HomstoZ}).

Our final example is the uniform Widom-Rowlinson model. Here $H=H_{\rm WR}$ is the graph on vertex set $\{1,2,3\}$ with all edges (and loops) present except the edge connecting $1$ and $3$. In the occupation and conditional occupation probability vectors we list the vertices in numerical order. Noting that ${\mathcal M}(H_{\rm WR})=\{(A,A),(B,B)\}$ where $A=\{1,2\}$ and $B=\{2,3\}$, we get the following via a routine calculation.
\begin{cor} \label{cor-WR-inf}
Fix $m \geq 2$ even. For all $x \in V$ we have
$$
d_\infty\left(\vec{v}_d(x),  \left(\frac{1}{4},\frac{1}{2},\frac{1}{4}\right)\right) \leq 2^{-\Omega(d)}.
$$
On the other hand, if $x, y \in {\mathcal E}$ then
$$
d_\infty\left(\vec{v}_d(x|f(y)=1), \left(\frac{1}{2},\frac{1}{2},0\right)\right) \leq 2^{-\Omega(d)}
$$
while if $x \in {\mathcal E}$ and $y \in {\mathcal O}$ then
$$
d_\infty\left(\vec{v}_d(x|f(y)=1), \left(0,\frac{1}{2}, \frac{1}{2}\right)\right) \leq 2^{-\Omega(d)}.
$$
\end{cor}

\section{Proofs of Theorems \ref{thm-structure-of-hom-space-coarse} and \ref{thm-structure-of-hom-space}}  \label{sec-cors}

We first note that if the weight set $\Lambda'$ is obtained from $\Lambda$ by multiplying each $\lambda_k$ by the same constant, then the distributions $p_\Lambda$ and $p_{\Lambda'}$ are identical. We may therefore assume without loss of generality that $\lambda_k \geq 1$ for all $k \in V(H)$.

Our main technical result (Theorem \ref{thm-small-prob-edge-not-ideal}) considers uniformly chosen homomorphisms, so to apply it to homomorphisms chosen according to $p_\Lambda$ we need to first relate $p_\Lambda$ to uniform distribution on a graph $H(\Lambda)$ built from $H$ and $\Lambda$. We use a technique introduced in \cite{BrightwellWinkler}.

Let $C=C(\Lambda)$ be the smallest integer such that $C\lambda_k$ is an integer for all $k \in V(H)$. For each $k$ let $S_k$ be an arbitrary set of size $C\lambda_k$, with the $S_k$'s disjoint. We construct $H(\Lambda)$ on vertex set $\cup_{k \in V(H)} S_k$ by joining $x$ and $y$ if and only if $x \in S_k$ and $y \in S_\ell$ for some $k\ell \in E(H)$. Equivalently, $H(\Lambda)$ is obtained from $H$ by replacing each vertex $k$ by a set of size $C\lambda_k$, each edge by a complete bipartite graph and each loop by a complete looped graph; see Figure \ref{fig-blowup}.
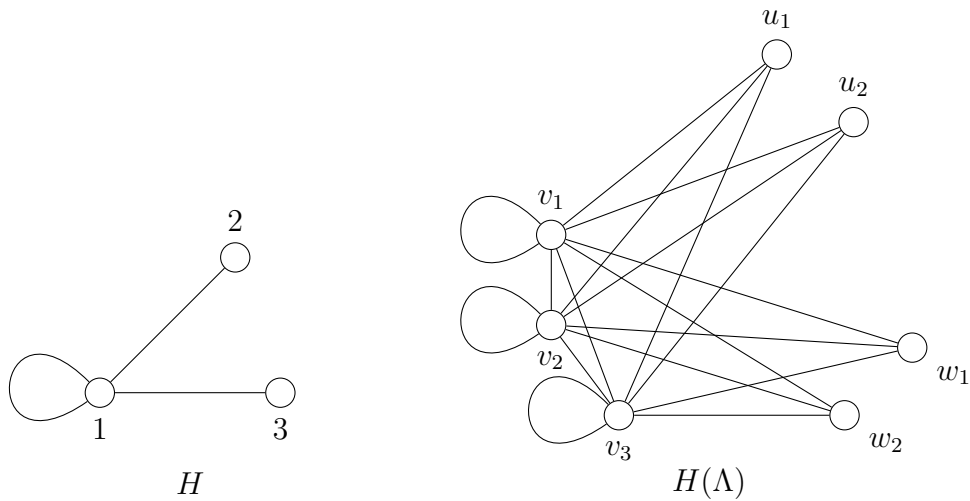
\begin{figure}[ht!] \center
\begin{tikzpicture}[scale=1.2,every loop/.style={}]
	\node (v1) at (1,1) [circle,draw,label=270:$1$] {};
	\node (v2) at (2.5,2.5) [circle,draw,label=90:$2$] {};
	\node (v3) at (3,1) [circle,draw,label=270:$3$] {};
	\node at (2,0) {$H$};
	
	\foreach \from/\to in {v1/v2,v1/v3}
	\draw (\from) -- (\to);
 	\path[every node/.style={font=\sffamily\small}]
        (v1)   edge[loop left, distance=.6in, out=135, in=215] node  {} (v1);

  \node (w1) at (6,2.75) [circle,draw,label=90:$v_1$] {};
  \node (w2) at (6,1.75) [circle,draw,label=268:$v_2$] {};
  \node (w21) at (6.75,.75) [circle,draw,label=270:$v_3$] {};
  \node (w3) at (10,1.5) [circle,draw,label=-30:$w_1$] {};
  \node (w31) at (9.25,.75) [circle,draw,label=-30:$w_2$] {};
  \node (w4) at (8.5,4.75) [circle,draw,label=90:$u_1$] {};
  \node (w5) at (9.35,4) [circle,draw,label=90:$u_2$] {};
%  \node (w6) at (10,3.35) [circle,draw,label=90:$u_3$] {};
  \node at (7.75,0) {$H(\Lambda)$};

  \foreach \from/\to in {w1/w2,w1/w3,w2/w3,w1/w4,w1/w5,w2/w4,w2/w5,w2/w21,w1/w21,w3/w21,w4/w21,w5/w21,w31/w1,w31/w2,w31/w21}
  \draw (\from) -- (\to);

 	\path[every node/.style={font=\sffamily\small}]
        (w1)   edge[loop, distance=.6in, out=135, in=215] node  {} (w1);
 	\path[every node/.style={font=\sffamily\small}]
        (w2)   edge[loop left, distance=.6in, out=135, in=215] node  {} (w2);
  \path[every node/.style={font=\sffamily\small}]
        (w21)   edge[loop left, distance=.6in, out=135, in=215] node  {} (w21);

\end{tikzpicture}
\caption{An example $H$ and $H(\Lambda)$ with $\lambda_1 = 3/2$, $\lambda_2 = 1$ and $\lambda_3 = 1$, so $C=2$.  Here $S_1 = \{v_1,v_2,v_3\}$, $S_2 = \{u_1,u_2\}$, and $S_3 = \{w_1,w_2\}$. \label{fig-blowup}}
\end{figure}

For each $f \in {\rm Hom}(\torusMd,H)$ let $A_f$ consist of those $g \in {\rm Hom}(\torusMd,H(\Lambda))$ with $g(v) \in S_{f(v)}$ for each $v \in V$. It is straightforward to verify that each $A_f$ satisfies $|A_f| = C^{m^d} w_\Lambda(f)$, and that the $A_f$'s form a partition of ${\rm Hom}(\torusMd,H(\Lambda))$. This implies that choosing an element $g$ uniformly from ${\rm Hom}(\torusMd,H(\Lambda))$ and then letting $f \in {\rm Hom}(\torusMd,H)$ be such that $g \in A_f$ is equivalent to choosing $f$ from ${\rm Hom}(\torusMd,H)$ according to $p_\Lambda$.

Before continuing, we note the following easily established correspondence between ${\mathcal M}(H(\Lambda))$ and ${\mathcal M}_\Lambda(H)$:
\begin{equation} \label{bw-corr}
\begin{array}{c}
|{\mathcal M}(H(\Lambda))|=|{\mathcal M}_\Lambda(H)| \\
\mbox{and} \\
(A',B') \in {\mathcal M}(H(\Lambda))~\mbox{if and only if} \\
A'=\cup_{k \in A} S_k~\mbox{and}~ B'=\cup_{\ell \in B} S_\ell~\mbox{for some}~(A,B) \in {\mathcal M}_\Lambda(H).
\end{array}
\end{equation}

Now let $g$ be chosen uniformly from ${\rm Hom}(\torusMd,H(\Lambda)$). By Theorem \ref{thm-small-prob-edge-not-ideal}, the expected number of non-ideal edges of $g$ is at most $(m-\Omega(1))^d$ and so by Markov's inequality there is a subset ${\rm Hom}'(\torusMd,H(\Lambda))$ of ${\rm Hom}(\torusMd,H(\Lambda))$ with
\begin{equation} \label{small-exceptional}
|{\rm Hom}'(\torusMd,H(\Lambda))| \geq \left(1-2^{-\Omega(d)}\right)|{\rm Hom}(\torusMd,H(\Lambda))|
\end{equation}
and with each $g \in {\rm Hom}'(\torusMd,H(\Lambda))$ having at most $(m-\Omega(1))^d$ non-ideal edges.

We now need an isoperimetric bound on the discrete torus. The following result is due to Bollob\'as and Leader \cite[Theorem 8]{BollobasLeader2}.
\begin{lemma} \label{lem-BL-iso}
Let $X \subseteq V$ satisfy $|X| \leq m^d/2$. The number of edges in $E$ which have exactly one vertex in common with $X$ is at least $|X|^{(d-1)/d}$.
\end{lemma}
We will use the following corollary.
\begin{cor} \label{cor-BL-iso}
Let $a$ satisfy $(m a)^{d/(d-1)} < 1/4$. If at most $m^d a$ edges are deleted from $\torusMd$ then the resulting graph has a component with at least $m^d(1-(m a)^{d/(d-1)})$ vertices.
\end{cor}

\begin{proof}
Let ${\mathcal D}$ be the set of deleted edges, and let $C_1, C_2, \ldots, C_k$ be the components of the graph on vertex set $V$ with edge set $E \setminus {\mathcal D}$, listed in order of increasing size (where size is measured by number of vertices). If $k=1$, we are done. Otherwise, let $X = \cup_{i=1}^\ell C_i$ where $\ell$ is chosen as large as possible so that $|X| \leq m^d/2$. Since ${\mathcal D}$ includes all of the edges which have exactly one vertex in common with $X$, we have by Lemma \ref{lem-BL-iso}
$$
m^d a \geq |{\mathcal D}| \geq |X|^{\frac{d-1}{d}}
$$
and so
$$
|X| \leq  m^d (m a)^{\frac{d}{d-1}} < m^d/4
$$
(the final inequality by hypothesis). By the definition of $\ell$, we have $|C_\ell|>m^d/4$. If $\ell=k-1$, we are done (since then $|C_\ell| \geq m^d(1-(m a)^{d/(d-1)})$). We complete the proof by arguing that we must have $\ell=k-1$. If not, let $X'$ be the union of all the components other than $C_{\ell+1}$ and those in $X$. By the same argument as above (since $|X'|\leq m^d/2$) we have $|X'| < m^d/4 < |C_\ell|$. This is a contradiction, since by our ordering of the components $X'$ is a union of components all at least as large as $C_\ell$.
\end{proof}

Corollary \ref{cor-BL-iso} implies that for each $g \in {\rm Hom}'(\torusMd,H(\Lambda))$ there is a collection ${\mathcal F}$ of edges which spans a connected subgraph of $\torusMd$ on at least  $m^d - (m-\Omega(1))^d$ vertices, and that all of these edges are ideal (note that in this application we have $a=2^{-\Omega(d)}$ and so certainly $(m a)^{d/(d-1)}<1/4$). By the connectivity of the subgraph induced by these edges, it follows that there is some $(A',B') \in {\mathcal M}(H(\Lambda))$ such that for each $uv \in {\mathcal F}$ with $u \in {\mathcal O}$, we have that $N_u$ is colored from $A'$ (and so in particular $v$ is) and $N_v$ is colored from $B'$ (and so in particular $u$ is). We may therefore decompose ${\rm Hom}'(\torusMd,H(\Lambda))$ as
$$
{\rm Hom}'(\torusMd,H(\Lambda)) = \cup_{(A',B') \in {\mathcal M}(H(\Lambda))} D(A',B')
$$
with the property that for each $g \in D(A',B')$ we can find a subset of $V$ of size at least $m^d - (m-\Omega(1))^d$ with each vertex of this set colored from $A'$ (resp. $B'$) if it is in ${\mathcal E}$ (resp. ${\mathcal O}$), and moreover all but at most $(m-\Omega(1))^d$ vertices of ${\mathcal O}$ (resp. ${\mathcal E}$) have all of $A'$ (resp. $B'$) appearing on their neighborhoods.

We now pass to a partition of ${\rm Hom}(\torusMd,H)$. For each $(A,B) \in {\mathcal M}_\Lambda(H)$, let $D_\Lambda(A,B)$ be the set of all $f \in {\rm Hom}(\torusMd,H)$ for which there is some $g \in A_f$ with $g \in D(A',B')$, where $(A',B')$ is obtained from $(A,B)$ by the correspondence described in (\ref{bw-corr}). The $D_\Lambda(A,B)$'s are disjoint, for if $f \in D_\Lambda(A,B)$ (with corresponding $g \in D(A',B')$) and $\tilde{f} \in D_\Lambda(\tilde{A},\tilde{B})$ (with corresponding $\tilde{g} \in D(\tilde{A}',\tilde{B}')$) with $(A,B)\neq (\tilde{A},\tilde{B})$, the neighborhoods of the endvertices of any edge which is ideal for both $g$ and $\tilde{g}$ witness that $f\neq \tilde{f}$.

Moreover, $D_\Lambda(A,B)$ inherits from $D(A',B')$ that for all $f \in D_\Lambda(A,B)$, the number of vertices $v \in {\mathcal E}$ (resp. ${\mathcal O}$) with $f(v) \not \in A$ (resp. $f(v) \not \in B$) is at most $(m-\Omega(1))^d$ (for concreteness, $(m-\kappa)^d$ for some $0 < \kappa < m$ that depends on $H$ and $\Lambda$ but may be chosen to be independent of $(A,B)$), and moreover all but at most $(m-\Omega(1))^d$ vertices $w$ of ${\mathcal O}$ (resp. ${\mathcal E}$) have the property that all colors from $A$ (resp. $B$) appear on $N_w$.

Set $D_\Lambda(0)={\rm Hom}(\torusMd,H) \setminus \cup_{(A,B)\in {\mathcal M}_\Lambda(H)} D_\Lambda(A,B)$. If $f \in D_\Lambda(0)$ then 
$$
A_f \subseteq {\rm Hom}(\torusMd,H(\Lambda)) \setminus {\rm Hom}'(\torusMd,H(\Lambda))
$$ 
and so by (\ref{small-exceptional})
$$
C^{m^d}w_\Lambda(D_\Lambda(0)) \leq 2^{-\Omega(d)}|{\rm Hom}(\torusMd,H(\Lambda))| = 2^{-\Omega(d)}C^{m^d}Z_\Lambda(\torusMd,H(\Lambda)).
$$
This completes the proof of Theorem \ref{thm-structure-of-hom-space-coarse}.

\medskip

We now turn to Theorem \ref{thm-structure-of-hom-space}. Our construction of the $C_\Lambda(A,B)$'s will be from scratch (and so in particular we will not refer to ideal edges); however, to establish the required properties of the $C_\Lambda(A,B)$'s we will relate them to the $D_\Lambda(A,B)$'s.

For each $(A,B) \in {\mathcal M}_\Lambda(H)$ we define a set $C_\Lambda(A,B)'$ as follows. First, for each $F_1 \subseteq {\mathcal E}$ and $F_2 \subseteq {\mathcal O}$ with $|F_1|+|F_2| \leq (m-\kappa)^d$ (with $\kappa$ as described in the construction of $D_\Lambda(A,B)$ above), let $C^{(F_1,F_2)}_\Lambda(A,B)'$ include all $f \in {\rm Hom}(\torusMd,H)$ for which every vertex of ${\mathcal E} \setminus F_1$ is colored from $A$, every vertex from $F_1$ is colored from $A^c$, every vertex of ${\mathcal O} \setminus F_2$ is colored from $B$, and every vertex from $F_2$ is colored from $B^c$ (note that for some choices of $(F_1, F_2)$ we may have $C^{(F_1,F_2)}_\Lambda(A,B)'=\emptyset$). Next, set
$$
C_\Lambda(A,B)' = \cup_{(F_1,F_2)} C^{(F_1,F_2)}_\Lambda(A,B)'.
$$

By our upper bound on $|F_1|+|F_2|$, we have $D_\Lambda(A,B) \subseteq C_\Lambda(A,B)'$ for each $(A,B)$. It is also clear that $|C_\Lambda(A,B)'|=|C_\Lambda(B,A)'|$ (because the mapping from ${\rm Hom}(\torusMd,H)$ to itself, induced by any automorphism of $\torusMd$ that maps ${\mathcal E}$ to ${\mathcal O}$, maps $C_\Lambda(A,B)'$ to $C_\Lambda(B,A)'$ bijectively, and is weight-preserving), and (for a similar reason) that if $\varphi(A)=\tilde{A}$ and $\varphi(B)=\tilde{B}$ for some weight-preserving automorphism $\varphi$ of $H$ then $|C_\Lambda(A,B)'|=|C_\Lambda(\tilde{A},\tilde{B})'|$. We do not yet have a partition of ${\rm Hom}(\torusMd,H)$, however, as the $C_\Lambda(A,B)'$'s  are not necessarily disjoint.

Most of the rest of the proof is devoted to establishing the following two facts. First, for each $(A,B) \in {\mathcal M}_\Lambda(H)$, $x \in {\mathcal E}$, $y \in {\mathcal O}$, $k \in A$ and $\ell \in B$, if $f$ is chosen from ${\rm Hom}(\torusMd,H)$ according to $p_\Lambda$ then
\begin{equation} \label{towards-5}
\begin{array}{rcl}
p_\Lambda(f(x)=k | f \in C_\Lambda(A,B)') & = & \frac{\left(1+2^{-\Omega(d)}\right)\lambda_k}{\lambda_A},\\
p_\Lambda(f(y)=\ell | f \in C_\Lambda(A,B)') & = & \frac{\left(1+2^{-\Omega(d)}\right)\lambda_\ell}{\lambda_B}.
\end{array}
\end{equation}
For the second, say that $f \in C_\Lambda(A,B)'$ is {\em balanced} if for each $k \in A$ (resp. $\ell \in B$) the proportion of vertices of ${\mathcal E}$ (resp. ${\mathcal O}$) colored $k$ (resp. $\ell$) is within a multiplicative factor $1 \pm (1-\kappa/(4m))^d$ of $\lambda_k/\lambda_A$ (resp. $\lambda_\ell/\lambda_B$). For all $(A,B) \in {\mathcal M}_\Lambda(H)$ we have the following:
\begin{equation} \label{towards-2}
p_\Lambda\left(f~\mbox{is not balanced} | f \in C_\Lambda(A,B)' \right) \leq \exp\left\{-\left(m-\frac{\kappa}{2}\right)^d/4\right\}.
\end{equation}

These two facts allow us to swiftly complete the proof of Theorem \ref{thm-structure-of-hom-space}. Indeed, for each $(A,B) \in {\mathcal M}_\Lambda(H)$, let $C_\Lambda(A,B)$ be the subset of $C_\Lambda(A,B)'$ consisting of balanced homomorphisms. The $C_\Lambda(A,B)$'s are clearly disjoint. Letting $C_\Lambda(0)$ be the complement of the union of the $C_\Lambda(A,B)$'s, we have that $w_\Lambda(C_\Lambda(0)) \leq 2^{-\Omega(d)}Z_\Lambda(\torusMd,H)$ since it consists of the unbalanced homomorphisms removed from the $C_\Lambda(A,B)'$'s (a collection with total weight at most $\exp\{-(m-\kappa/2)^d/4\}Z_\Lambda(\torusMd,H)$, by (\ref{towards-2})) together with some subset of $D_\Lambda(0)$ (with total weight at most $2^{-\Omega(d)}Z_\Lambda(\torusMd,H)$). This establishes that our partition satisfies statement \ref{main-thm-item-1} of Theorem \ref{thm-structure-of-hom-space}.

Statement \ref{main-thm-item-2} is immediate from the construction of the $C_\Lambda(A,B)$'s. Statements \ref{main-thm-item-3} and \ref{main-thm-item-4} follow from the corresponding statements for the $C_\Lambda(A,B)'$'s, since the sizes of $C_\Lambda(A,B)'$ and $C_\Lambda(A,B)$ differ by a multiplicative factor of no more than $1\pm 2^{-\Omega(d)}$. Finally, statement \ref{main-thm-item-5} follows from (\ref{towards-5}) for the same reason.

We now begin the verification of (\ref{towards-5}) and (\ref{towards-2}), beginning with (\ref{towards-5}). Fix $(A,B) \in {\mathcal M}_\Lambda(H)$, $x \in {\mathcal E}$ and $k \in A$ (the case $y \in {\mathcal O}$ and $\ell \in B$ is analogous). If $(F_1,F_2)$ is such that $x \not \in F_1 \cup N(F_2)$, then since $x$ is adjacent to vertices colored from $B$, and all vertices of $A$ are adjacent to all vertices of $B$, we have the following: for $f$ chosen from $C^{(F_1,F_2)}_\Lambda(A,B)'$ according to $p_\Lambda$, the probability that $f(x)=k$ is exactly $\lambda_k/\lambda_A$. Thus (\ref{towards-5}) will follow if we can show that the contribution to $w_\Lambda(C_\Lambda(A,B)')$ from those $C_\Lambda^{(F_1,F_2)}(A,B)'$'s with $x \in F_1 \cup N(F_2)$ is at most $2^{-\Omega(d)}w_\Lambda(C_\Lambda(A,B)')$. To establish this, note that
\begin{eqnarray*}
& & \sum_{(F_1,F_2)} w_\Lambda(C^{(F_1,F_2)}_\Lambda(A,B)'){\bf 1}_{\{x \in F_1 \cup N(F_2)\}} \\
& = & \frac{1}{m^d}\sum_{y \in {\mathcal E}} \sum_{(F_1,F_2)} w_\Lambda(C^{(F_1,F_2)}_\Lambda(A,B)'){\bf 1}_{\{y \in F_1 \cup N(F_2)\}} \\
& \leq & \frac{1}{m^d}\sum_{(F_1,F_2)} |F_1 \cup N(F_2)|w_\Lambda(C^{(F_1,F_2)}_\Lambda(A,B)') \\
& \leq & \frac{(2d+1)(m-\kappa)^d}{m^d} w_\Lambda(C_\Lambda(A,B)'.
\end{eqnarray*}
The first equality follows from the symmetry of both $\torusMd$ and the construction of $C_\Lambda(A,B)'$. In the first inequality we reverse the order of summation, and in the second we bound $|F_1 \cup N(F_2)|$ by $(2d+1)(m-\kappa)^d$.

Now we turn to (\ref{towards-2}). Again fix $(A,B) \in {\mathcal M}_\Lambda(H)$.
A lower bound on $w_\Lambda(C^{(F_1,F_2)}_\Lambda(A,B)')$ (for $C^{(F_1,F_2)}_\Lambda(A,B)' \neq \emptyset$) is
\begin{equation} \label{eq-lb}
\lambda_A^{m^d/2-|F_1\cup N(F_2)|}\lambda_B^{m^d/2-|F_2\cup N(F_1)|}.
\end{equation}
As before, this is because every vertex in ${\mathcal E} \setminus F_1 \cup N(F_2)$ is adjacent only to vertices colored only from $B$ and so may be given any color from $A$, with a similar argument for vertices from ${\mathcal O} \setminus F_2 \cup N(F_1)$ (note that in this lower bound we are using the assumption $\lambda_i \geq 1$ for all $i$).

For $\delta >0$, an upper bound on the sum of the weights of those $f \in C^{(F_1,F_2)}_\Lambda(A,B)'$ in which a particular color $k$ from $A$ appears either on a proportion less than $(\lambda_k/\lambda_A - \delta)$ of ${\mathcal E}$, or on a proportion greater than $(\lambda_k/\lambda_A + \delta)$, is
\begin{equation} \label{eq-ub}
\left(\sum_{{i \leq (\lambda_k/\lambda_A - \delta)m^d/2 \atop{i \geq (\lambda_k/\lambda_A + \delta)m^d/2}}} {m^d/2 \choose i}(\lambda_A-\lambda_k)^{m^d/2-i}\lambda_k^i\right) \lambda_B^{m^d/2}  \lambda_{V(H)}^{|F_1\cup N(F_2)|+|F_2\cup N(F_1)|}.
\end{equation}
By standard Binomial concentration inequalities (see for example \cite{Hoeffding} or \cite[Appendix A]{AlonSpencer}, we have
\begin{equation} \label{Bern}
\sum_{{i \leq (\lambda_k/\lambda_A - \delta)m^d/2 \atop{i \geq (\lambda_k/\lambda_A + \delta)m^d/2}}} {m^d/2 \choose i}(\lambda_A - \lambda_k)^{m^d/2-i}\lambda_k^i \leq 2\exp\left\{-\delta^2 m^d/2\right\} \lambda_A^{m^d/2}.
\end{equation}
Combining (\ref{eq-lb}), (\ref{eq-ub}) and (\ref{Bern}) we find that for $f$ chosen from non-empty $C^{(F_1,F_2)}_\Lambda(A,B)'$ according to $p_\Lambda$, the probability that a particular color appears either on a proportion less than $(\lambda_k/\lambda_A - \delta)$ of ${\mathcal E}$ or on a proportion greater than $(\lambda_k/\lambda_A + \delta)$ is at most
$$
\frac{2\lambda_{V(H)}^{2|F_1\cup N(F_2)|+2|F_2\cup N(F_1)|}}{\exp\left\{\delta^2 m^d/2\right\}} \leq \exp\left\{-\delta^2 m^d/2 + O(d(m-\kappa)^d)\right\}
$$
(again using $\lambda_i \geq 1$ for all $i$ as well as our upper bound on $|F_1|+|F_2|$).
Repeating this argument for colors from $B$ and applying the law of total probability and a union bound, we find that for $f$ chosen from $C_\Lambda(A,B)'$ according to $p_\Lambda$, the probability that either there is some color $k$ from $A$ which fails to appear on a proportion between $(\lambda_k/\lambda_A - \delta)$ and $(\lambda_k/\lambda_A + \delta)$ of ${\mathcal E}$, or there is some color $\ell$ from $B$ which fails to appear on a proportion between $(\lambda_\ell/\lambda_B - \delta)$ and $(\lambda_\ell/\lambda_B + \delta)$ of ${\mathcal O}$ is at most $\exp\{-\delta^2 m^d/2 + O(d(m-\kappa)^d)\}$. Taking $\delta = (1-\kappa/(4m))^d$ gives the required result.

\section{Proof of Theorem \ref{thm-small-prob-edge-not-ideal}} \label{sec-tech-proof}

Our strategy is to put an upper bound on the entropy of a uniformly chosen element of ${\rm Hom}(\torusMd,H)$ that is smaller than a trivial lower bound unless $\varepsilon$ is suitably small. We build on ideas introduced by Kahn \cite{Kahn2}.

\subsection{Entropy} \label{sec-entropy}

In this section we very briefly review the entropy material that is relevant for the proof of Theorem \ref{thm-small-prob-edge-not-ideal}. See \cite{Kahn2} for an expanded treatment appropriate to the present application, or for example \cite{McEliece} for a very thorough discussion.  In what follows, $X, Y$, etc. are discrete random variables, taking values in any finite set. Throughout, we take $\log = \log_2$.

The (binary) entropy function is $H(\alpha) = -\alpha \log \alpha - (1-\alpha) \log (1-\alpha)$.  The \emph{entropy} of the random variable $X$ is
$H(X) = \sum_{x} -p(x) \log p(x)$ where we write $p(x)$ for $\Pr(X = x)$ (and later $p(x|y)$ for $\Pr(X = x|Y=y)$). The inequality that makes entropy a useful tool for counting is
\begin{equation} \label{ent-uniform}
H(X) \leq \log|{\rm range}(X)|,
\end{equation}
with equality if and only if $X$ is uniform. For random variables $X, Y$ and $Z$ where $Y$ determines $Z$, we also have
\begin{equation} \label{ent-drop-cond}
H(X | Y) \leq H(X) ~~~\mbox{and} ~~~ H(X | Y) \leq H(X | Z),
\end{equation}
that is, dropping or lessening conditioning does not decrease entropy (here $H(X|Y)=\sum_{y} p(y)\sum_{x} -p(x|y) \log p(x|y)$ is a conditional entropy). We will also use the (conditional) chain rule: for $X=(X_1, \ldots, X_n)$ a random vector,
\begin{equation} \label{ent-chain-cond}
H(X | Y) = H(X_1 | Y) + H(X_2 | X_1, Y) + \cdots + H(X_n | X_1,\ldots,X_{n-1}, Y).
\end{equation}
Finally, we will need the conditional version of Shearer's lemma from \cite{Kahn2} (extending the original Shearer's lemma from \cite{ChungFranklGrahamShearer}).  For a random vector $X = (X_1, \ldots, X_m)$ and $A \subseteq [m]:=\{1, \ldots, m\}$, set $X_A = (X_i : i \in A)$.
\begin{lemma} \label{lem-Shearer}
Let $X = (X_1,\ldots,X_m)$ be a random vector and ${\mathcal A}$ a collection of subsets (possibly with repeats) of $[m]$, with each element of $[m]$ contained in at least $t$ members of ${\mathcal A}$.  Then, for any partial order $\prec$ on $[m]$,
$$
H(X) \leq \frac{1}{t} \sum_{A \in {\mathcal A}} H(X_A | (X_i : i \prec A)),
$$
where $i \prec A$ means $i \prec a$ for all $a \in A$.
\end{lemma}

\subsection{Notation and definitions} \label{notation}

It will be convenient to gather together all of our notation in a single place. For whatever graph is under discussion, we use $\sim$ to indicate adjacency of pairs of vertices. For $A, B \subseteq V(H)$ write $A \sim B$ if $a \sim b$ for all $a \in A$ and $b \in B$. For $v \in V$ set $N_v = \{w \in V \colon w \sim v\}$. Recall that
$$
\eta(H) = \max\{|A||B| \colon A,B \subseteq V(H), \,A \sim B\}
$$
and
$$
{\mathcal M}(H) = \{(A,B) \colon A,B \subseteq V(H),\, A \sim B,\, |A||B| = \eta(H)\}.
$$
Define
$$
{\mathcal S}(H) = \{A \colon (A,B) \in {\mathcal M}(H) ~\mbox{for some}~B\}.
$$
For $A \subseteq V(H)$ let $n(A) = \{v \in V(H) \colon \{v\} \sim A\}$, and for $A, B \subseteq V(H)$ let $p(A,B)$ be the number of pairs $(a,b) \in A \times B$ with $a \nsim b$.
Let
$$
V^\star = \left\{x=(x_1, \ldots, x_d) \in V \colon x_d=0,\,x \in {\mathcal E}\right\}
$$
(a set of size $m^{d-1}/2$). For each $v \in V^\star$ set
$$
{\mathcal C}(v) = \{ v + (0,\ldots,0, i) \colon 0 \leq i \leq m-1\}.
$$
In other words, ${\mathcal C}(v)$ is the set of all vertices in $V$ which agree with $v$ on the first $d-1$ coordinates; note that unless $m=2$, ${\mathcal C}(v)$ induces a cycle in $\torusMd$. (In the case $m=2$, ${\mathcal C}(v)$ simply induces an edge; this slight difference between $m=2$ and $m \geq 4$ is something that has to be accommodated throughout the proof.) Throughout the proof we think of ${\mathcal C}(v)$ as an \emph{ordered} tuple of vectors $(v_0,v_1,\ldots,v_{m-1})$ with each $v_i = v + (0, \ldots, 0, i)$.

For $u \in {\mathcal C}(v)$ for some $v \in V^\star$, let $u'_{+} = u+(0,\ldots,0,1)$ and $u'_{-} = u-(0,\ldots,0,1)$ (so $u'_{+}=u'_{-}$ if and only if $m = 2$), and set
$$
M_u = N_u \setminus \{u'_{+},u'_{-}\}
$$
and
$$
M_{{\mathcal C}(v)} = M_{v_0} \cup \cdots \cup M_{v_{m-1}}.
$$
A key observation that drives our proof is that the subgraph of $\torusMd$ induced by $M_{{\mathcal C}(v)}$ is a disjoint union of $2d-2$ cycles of length $m$ (when $m \geq 4$) or of $d-1$ disjoint edges (when $m=2$); this significantly restricts the appearance of an $H$-coloring on $M_{{\mathcal C}(v)}$ given its appearance on ${\mathcal C}(v)$.

To each $v \in V^\star$ with $|v| \geq 2m$ (where $|\cdot|$ indicates the sum of the coordinates) associate a $w(v) \in V^\star$ with $|w(v)|=|v|-2m$ and with $w(v) < v$ in the usual component-wise partial order on ${\mathbb Z}^d$. For $|v| < 2m$ we do not define a $w(v)$, but it will prove convenient to adopt the convention in this case that $M_w=\emptyset$. From now on, whenever $w$ appears, it will be $w(v)$ for whatever $v \in V^\star$ is under consideration.

We will use $\sumrange$ to indicate a tuple with each $A_i \subseteq V(H)$, and when $\sumrange$ appears as a range of summation it will vary over all possible such tuples. We will use $\altAB$ for the tuple $(A,B,\ldots, A,B)$, and $n\sumrange$ for the tuple $(n(A_0), \ldots, n(A_{m-1}))$. We denote by $g\sumrange$ the number of ways of choosing $(x_0, \ldots, x_{m-1})$ with $x_i \in A_i$ for each $i$ and with $x_0 \sim \cdots \sim x_{m-1} \sim x_0$ (that is, with the $x_i$'s, taken consecutively, forming a cycle).

\subsection{Events and probabilities} \label{probs}

Now let $f$ be uniformly chosen from ${\rm Hom}(\torusMd,H)$. We define a number of events in the associated probability space.
For $A \subseteq V(H)$ and $v \in V^\star$, let
$$
Q_{v,A} = \{ f(N_v) = A\},
$$
$$
R_{v,A} = \{ f(M_v) = A\},
$$
$$
Q_{{\mathcal C}(v), \sumrange} = \cap_{i=0}^{m-1} Q_{{v_i,A_i}}
$$
and
$$
R_{{\mathcal C}(v), \sumrange} = \cap_{i=0}^{m-1} R_{{v_i,A_i}}.
$$
To denote the probability of each of these events, we will replace the leading upper case letter with the corresponding lower case letter; so, for example,
$$
q_{v,A} = \Pr\left(Q_{v,A}\right).
$$
For $u \in {\mathcal C}(v)$ for some $v \in V^\star$ let $R_u=\{f(y):y \in M_u\}$ be the random variable indicating the palette of colors used on $M_u$, and let
$$
T_{{\mathcal C}(v)} = (R_{v_0},\ldots,R_{v_{m-1}}).
$$

Finally, define $\varepsilon$ (depending on $d$, $m$ and $H$, but by the symmetry of $\torusMd$ independent of $v$) by
$$
1 - \varepsilon = \sum_{(A,B) \in {\mathcal M}(H)} r_{{\mathcal C}(v), \altAB}.
$$

\subsection{A partial order on $V$} \label{partial}

For $0 \leq k \leq (m-1)(d-1)$, let
$$
L_k = \left\{x \in V \colon \sum_{i=1}^{d-1} x_i = k\right\}.
$$
% made the aim of the partial order more explicit
We refer to the $L_k$'s as the {\em levels} of $V$; note that they partition $V$. Following the approach of \cite{Kahn2}, we wish to put a partial order on $V$ that satisfies (\ref{prec-prop-1}) and (\ref{prec-prop-2}) below. We will achieve this by putting an order $\prec$ on the indices of the levels, as follows. Begin by ordering the odd natural numbers in the usual order, up to $m-1$. Next put $0$, then $m+1$, then $2$, then $m+3$, etc., interleaving the standard order of the evens and the odds. This order for $m=2$ is used in \cite{Kahn2}, and begins $1 \prec 0 \prec 3 \prec 2 \prec 5 \prec 4 \prec \cdots$. For $m=4$, it begins $1 \prec 3 \prec 0 \prec 5 \prec 2 \prec 7 \prec 4 \prec \cdots$, and for $m=6$ it begins $1 \prec 3 \prec 5 \prec 0 \prec 7 \prec 2 \prec 9 \prec 4 \prec \cdots$.

For each even $i \in {\mathbb N}$ let $X_i=\{i-m+1, i-1, i+1, i+m-1\}\cap {\mathbb N}$ (or $\{i-1,i+1\}\cap {\mathbb N}$ if $m=2$) and $Y_i = \{i-3m+1, i-2m-1,i-2m+1, i-m-1\}\cap {\mathbb N}$ (or $\{i-5,i-3\}\cap {\mathbb N}$ if $m=2$). The order $\prec$ is constructed specifically to satisfy that $x \prec i$ for all $x \in X_i$ and $y \prec x$ for all $x \in X_i$ and $y \in Y_i$.

We use $\prec$ to obtain a partial order (which we shall also call $\prec$) on $V$ by declaring $x \prec y$ if and only if $i \prec j$, where $x \in L_i$ and $y \in L_j$. This partial order has two properties that will be critically important for us. For the first of these, note that for $v \in V^\star$, if $v \in L_i$ for some $i$ (necessarily even), then ${\mathcal C}(v) \subseteq L_i$ and $M_{{\mathcal C}(v)} \subseteq \cup_{x \in X_i} L_x$, and so
\begin{equation} \label{prec-prop-1}
M_{{\mathcal C}(v)}  \subseteq  \{x \colon x\prec {\mathcal C}(v)\}.
\end{equation}
For the second property, note that since $M_w \subseteq \cup_{y \in Y_i} L_y$ for $v \in L_i$ we have
\begin{equation} \label{prec-prop-2}
M_w  \subseteq \{x \colon x \prec M_{{\mathcal C}(v)} \}.
\end{equation}

\subsection{The proof of Theorem \ref{thm-small-prob-edge-not-ideal}} \label{subsec-proof}

We will show that $\varepsilon < 2^{-\Omega(d)}$ (with the implicit constant depending on $m$ and $H$). From this, Theorem \ref{thm-small-prob-edge-not-ideal} follows. To see this, first observe that for $(A,B) \in {\mathcal M}(H)$ we have
$Q_{{\mathcal C}(v),\altAB} \supseteq R_{{\mathcal C}(v),\altAB}$. Indeed, consider any $f \in R_{{\mathcal C}(v),\altAB}$. For each even $i$ we must have $f(v_i) \sim a$ for all $a \in A$, and so since $(A,B) \in {\mathcal M}(H)$, we must have $f(v_i) \in B$; similarly, for odd $i$ we must have $f(v_i) \in A$. It follows that
$$
1-\varepsilon \leq \sum_{(A,B) \in {\mathcal M}(H)} q_{{\mathcal C}(v),\altAB}.
$$
Now let $e=xy$ be an edge of $\torusMd$; by symmetry we may assume that $e=v_0v_1$ for some $v = v_0 \in V^\star$. The event that $e$ is ideal contains the event $\cup_{(A,B) \in {\mathcal M}(H)} Q_{{\mathcal C}(v),\altAB}$ (a union of disjoint events), and so the probability that $e$ is ideal is at least $1-\varepsilon$.

To bound $\varepsilon$ we consider the entropy $H(f)$ of an $f \in {\rm Hom}(\torusMd,H)$, chosen uniformly. We first put a trivial lower bound on $H(f)$:
\begin{equation} \label{Hlowerbound}
H(f) = \log |{\rm Hom}(\torusMd,H)| \geq \frac{m^d}{2} \log\eta(H),
\end{equation}
the equality from (\ref{ent-uniform}) and the inequality obtained by choosing any $(A,B) \in {\mathcal M}(H)$ and considering only pure-$(A,B)$ colorings (as defined in Section \ref{sec-intro}). The bulk of the proof will be devoted to finding an upper bound on $H(f)$ which, for $\varepsilon$ too large, is smaller than this trivial lower bound.

We will upper bound $H(f)$ by an application of Shearer's lemma (with conditioning), that is, Lemma \ref{lem-Shearer}. For $m\geq 4$, we take as our covering family $\{M_{{\mathcal C}(v)} \colon v \in V^\star\}$ together with $2d-2$ copies of ${\mathcal C}(v)$ for each $v \in V^\star$. For $m=2$ we take $\{M_{{\mathcal C}(v)} \colon v \in V^\star\}$ together with $d-1$ copies of ${\mathcal C}(v)$ for each $v \in V^\star$. Each vertex of $\torusMd$ is covered $2d-2$ times by this family (in the case $m\geq 4$) or $d-1$ times (in the case $m=2$) and so, bearing (\ref{ent-drop-cond}), (\ref{prec-prop-1}) and (\ref{prec-prop-2}) in mind we have
\begin{equation}
\label{HShearer} H(f)  \leq \sum_{v \in V^\star} H(f\rst{{\mathcal C}(v)} | f\rst{M_{{\mathcal C}(v)}}) +  \left(\frac{1+{\mathbf 1}_{\{m=2\}}}{2d-2}\right)\sum_{v \in V^\star} H(f\rst{M_{{\mathcal C}(v)}} | f\rst{M_w}),
\end{equation}
where $f\rst S$ denotes the restriction of $f$ to the set $S\subseteq V$ (note that this is our only use of the order $\prec$).
For the first term on the right-hand side of (\ref{HShearer}) we expand out the conditional entropy and use (\ref{ent-uniform}) to get
\begin{eqnarray}
\nonumber & & H(f\rst{{\mathcal C}(v)} | f\rst{M_{{\mathcal C}(v)}}) \\
\nonumber & \leq & \sum_{\sumrange} r_{{\mathcal C}(v), \sumrange} H\left(f({\mathcal C}(v)) | \left\{T_{{\mathcal C}(v)} = \sumrange\right\}\right) \\
\label{Hfvfz} & \leq & \sum_{\sumrange}  r_{{\mathcal C}(v),\sumrange} \log\left(g(n\sumrange)\right).
\end{eqnarray}

We now turn to the second term on the right-hand side of (\ref{HShearer}). For $|v|\leq 2m-1$ we use (\ref{ent-uniform}) to naively bound
\begin{equation}
\label{small-v} H(f\rst{M_{{\mathcal C}(v)}} | f\rst{M_w}) \leq \left(\frac{2d-2}{1+{\bf 1}_{\{m=2\}}}\right)m \log |V(H)|;
\end{equation}
this will ultimately not be too costly since there are not too many such $v$. Specifically, the number of such $v$ is exactly the number of vectors $(a_1, \ldots, a_{d-1}) \in \{0,\ldots, m-1\}^{d-1}$ with $\sum_{i=0}^{d-1} a_i \leq 2m-2$ and even; this is at most the number of solutions to $\sum_{i=0}^d a_i = 2m-2$ in non-negative integers, which is at most ${2m+d-3 \choose 2m-2}$.

For $|v|\geq 2m$ we use (\ref{ent-drop-cond}) and (\ref{ent-chain-cond}) to obtain
\begin{eqnarray}
\nonumber H(f\rst{M_{{\mathcal C}(v)}} | f\rst{M_w}) & \leq & H(f\rst{M_{{\mathcal C}(v)}} | R_w)\\
\nonumber & = &  H(f\rst{M_{{\mathcal C}(v)}}, T_{{\mathcal C}(v)} | R_w)\\
\label{H2Shearer} & \leq & H(T_{{\mathcal C}(v)} | R_w) + H(f\rst{M_{{\mathcal C}(v)}} | T_{{\mathcal C}(v)}),
\end{eqnarray}
the equality holding since $f\rst{M_{{\mathcal C}(v)}}$ determines $T_{{\mathcal C}(v)}$. For the second term on the right hand side of (\ref{H2Shearer}) we expand out the conditional entropy and the use (\ref{ent-uniform}) to get
\begin{eqnarray}
\nonumber & & H(f\rst{M_{{\mathcal C}(v)}} | T_{{\mathcal C}(v)}) \\
\nonumber & = & \sum_{\sumrange} r_{{\mathcal C}(v),\sumrange} H(f\rst{M_{{\mathcal C}(v)}} | \{T_{{\mathcal C}(v)}=\sumrange\}) \\
 \label{1HXv} & \leq & \sum_{\sumrange} r_{{\mathcal C}(v), \sumrange} \left(\frac{2d-2}{1+{\mathbf 1}_{\{m=2\}}}\right)\log(g\sumrange).
\end{eqnarray}
Here we use that $M_{{\mathcal C}(v)}$ consists of $2d-2$ disjoint cycles (in the case $m\geq 4$) or $d-1$ disjoint edges (in the case $m=2$).

Inserting (\ref{Hfvfz}), (\ref{small-v}), (\ref{H2Shearer}) and (\ref{1HXv}) into (\ref{HShearer}), combining with (\ref{Hlowerbound}), summing over $v \in V^\star$ (noting that $|V^\star|=m^{d-1}/2$) and using the symmetry of $\torusMd$ we obtain
\begin{eqnarray}
\nonumber & & m \log \eta(H) \\
\label{finally-inq} & \leq & \frac{2{2m+d-3 \choose 2m-2}\log|V(H)|}{m^{d-2}} + \left(\frac{1+{\bf 1}_{\{m=2\}}}{2d-2}\right) H(T_{{\mathcal C}(v)} | R_w) \\
\nonumber & & + \sum_{\sumrange} r_{{\mathcal C}(v),\sumrange} \log\left(g\sumrange g(n\sumrange)\right).
\end{eqnarray}

We now focus on the sum on the right-hand side of (\ref{finally-inq}). Using the trivial bound
\begin{equation}
\label{g-triv} g\sumrange \leq \prod_{i=0}^{m-1} |A_i|
\end{equation}
together with the observation that for any $(A,B) \in {\mathcal M}(H)$ we have $n(A)=B$ and $n(B)=A$, we have
\begin{equation}
\label{good-AB} g(\altAB)g(n(\altAB)) \leq \eta(H)^m
\end{equation}
for any such $(A,B)$ (actually we have equality in (\ref{good-AB}), but we will not need it). On the other hand, we claim that if $\sumrange$ is not of the form $\altAB$ for some $(A,B) \in {\mathcal M}(H)$ then there is a constant $\delta(H) \geq 1$ such that
\begin{equation}
\label{bad-AB} g\sumrange g(n\sumrange) \leq \eta(H)^m - \delta(H).
\end{equation}
To see this, note first that if there is an $A \in \sumrange$ with $A \not \in {\mathcal S}(H)$, $A_0$ say, then from (\ref{g-triv}) we have
$$
g\sumrange  g(n\sumrange) \leq \prod_{i=0}^{m-1} |A_i||n(A_i)|,
$$
and since each of the terms in the product above is at most $\eta(H)$, and one ($|A_0||n(A_0)|$) is strictly less than $\eta(H)$, we get (\ref{bad-AB}). So we may assume that $\sumrange \in {\mathcal S}(H)^m$, but is not of the form $\altAB$. Since $(A,B)\in{\mathcal M}(H)$ is equivalent to $A, B \in {\mathcal S}(H)$ and $A=n(B)$, $B=n(A)$, we may assume without loss of generality that $A_1 \neq n(A_0)$. We have
$$
g\sumrange   \leq (|A_0||A_1|-p(A_0,A_1)) \prod_{i=2}^{m-1}  |A_i|
$$
and
$$
g(n\sumrange) \leq (|n(A_0)||n(A_1)|-p(n(A_0),n(A_1))) \prod_{i=2}^{m-1}  |n(A_i)|.
$$
If one of $p(A_0,A_1)$, $p(n(A_0),n(A_1))$ is non-zero, then as before the product of these two bounds is strictly less than $\eta(H)^m$, giving (\ref{bad-AB}) in this case. If they are both $0$ then we have $A_0 \sim A_1$ and $n(A_0) \sim n(A_1)$, so $A_1 \subseteq n(A_0)$ and $n(A_0) \subseteq A_1$, so $A_1=n(A_0)$, a contradiction.

Recalling the definition of $\varepsilon$, together (\ref{good-AB}) and (\ref{bad-AB}) yield
\begin{eqnarray*}
& & \sum_{\sumrange} r_{{\mathcal C}(v),\sumrange} \log \left(g\sumrange  g(n\sumrange)\right) \\
& \leq & \varepsilon \log(\eta(H)^m-\delta(H)) + (1-\varepsilon)\log\eta(H)^m \\
 & = & m \log\eta(H) +\varepsilon \log \left(1-\frac{\delta(H)}{\eta(H)^m}\right)\\
 & \leq & m \log\eta(H) -\frac{\varepsilon \delta(H) \log e}{\eta(H)^m}
\end{eqnarray*}
(recall $\log=\log_2$).
Inserting into (\ref{finally-inq}) we get
\begin{equation}\label{KeyEquation}
\frac{\varepsilon \delta(H) \log e}{\eta(H)^m} \leq \frac{2{2m+d-3 \choose 2m-2}\log|V(H)|}{m^{d-2}} + \left(\frac{1+{\bf 1}_{\{m=2\}}}{2d-2}\right) H(T_{{\mathcal C}(v)} | R_w).
\end{equation}

The final entropy term we need to analyze is $H(T_{{\mathcal C}(v)} | R_w)$. A naive upper bound from (\ref{ent-uniform}) is
$$
H(T_{{\mathcal C}(v)} | R_w) \leq |V(H)|m,
$$
the right-hand side being the logarithm of the size of the range of possible values. Inserting this into (\ref{KeyEquation}) we have
\begin{equation} \label{epsilon-first-shot}
\frac{\varepsilon \delta(H) \log e}{\eta(H)^m} \leq \frac{2{2m+d-3 \choose 2m-2}\log|V(H)|}{m^{d-2}} + \left(\frac{1+{\bf 1}_{\{m=2\}}}{2d-2}\right)|V(H)|m,
\end{equation}
showing that $\varepsilon \leq c/d$ for some constant $c$ depending on $H$ and $m$.

The information that $\varepsilon = o(1)$ as $d \rightarrow \infty$ allows us to strengthen our bound on $H(T_{{\mathcal C}(v)} | R_w)$, via the following key lemma.
\begin{lemma}\label{Hlemma}
For any $(A,B) \in {\mathcal M}(H)$,
$$
\Pr(R_{{\mathcal C}(v),\altAB}   | R_{w,A}) \geq 1- \frac{(3m-1)\varepsilon}{r_{w,A}},
$$
and also
$$
\sum_{A \notin {\mathcal S}(H)} r_{w,A} \leq \varepsilon.
$$
\end{lemma}

\begin{proof}
Choose $w_1, \ldots, w_{2m-1} \in V^\star$ with $w < w_1 < \cdots <w_{2m-1} < v$ in the usual partial ordering of ${\mathbb Z}^d$.  Then
\begin{eqnarray*}
\left(R_{{\mathcal C}(v),\altAB}\right)^c \cap R_{w,A} & \subset & \left(R_{w,A} \cap \left(R_{w_1,B}\right)^c\right) \cup \left(R_{w_1,B} \cap \left(R_{w_2,A}\right)^c\right) \cup \cdots \\
& & \cup \left(R_{w_{2m-1},B} \cap \left(R_{v_0,A}\right)^c\right) \cup \left(R_{v_0,A} \cap \left(R_{v_1,B}\right)^c\right) \cdots \\
& & \cup \left(R_{v_{m-2},A} \cap \left(R_{v_{m-1},B}\right)^c\right),
\end{eqnarray*}
and each of the $3m-1$ events on the right hand side occurs with probability less that $\varepsilon$, by symmetry of $\torusMd$.
Therefore
\begin{eqnarray*}
\Pr(\left(R_{{\mathcal C}(v),(A,B)}\right)^c | R_{w,A}) & = & \frac{\Pr\left(\left(R_{{\mathcal C}(v),(A,B)}\right)^c \cap R_{w,A}\right)}{r_{w,A}} \\
& \leq & \frac{(3m-1)\varepsilon}{r_{w,A}}.
\end{eqnarray*}
Also, $r_{w,A} \geq r_{{\mathcal C}(w), \altAB}$ implies
$$
\sum_{A \in {\mathcal S}(H)} r_{w,A} \geq \sum_{A \in {\mathcal S}(H)} r_{{\mathcal C}(w),\altAB} = \sum_{(A,B) \in {\mathcal M}(H)} r_{{\mathcal C}(w),\altAB}  = 1-\varepsilon.
$$
\end{proof}

We now partition ${\mathcal S}(H)$ by ${\mathcal S}(H) = {\mathcal S}_1(H) \cup {\mathcal S}_2(H)$, where $A \in {\mathcal S}_1(H)$ if and only if $r_{w,A} \leq 2(3m-1)\varepsilon$ (note that this partition depends on $d$ as well as on $H$, and for fixed $m$ and $H$ it may change for different values of $d$). For convenience we also write ${\mathcal S}_0(H)$ for the complement of ${\mathcal S}(H)$ (in the power set of $V(H)$). Expanding out the conditional entropy we have
$$
H(T_{{\mathcal C}(v)} | R_w) = \sum_{i=0}^2 \sum_{A \in {\mathcal S}_i(H)} r_{w,A} H(T_{{\mathcal C}(v)} | R_{w,A}).
$$
Trivially (from (\ref{ent-uniform}) and the second statement of Lemma \ref{Hlemma}),
\begin{equation}
\label{last-en-term-1} \sum_{A \in {\mathcal S}_0(H)} r_{w,A} H(T_{{\mathcal C}(v)} | R_{w,A}) \leq \varepsilon |V(H)| m.
\end{equation}
For the remaining two terms of the sum, we need to do a little groundwork. For each $A$, $-H(T_{{\mathcal C}(v)} | R_{w,A})$ is the sum over all $\sumrange$ of
$$
\Pr(\{T_{{\mathcal C}(v)}=\sumrange\} | R_{w,A}) \log \left(\Pr(\{T_{{\mathcal C}(v)}=\sumrange\} | R_{w,A})\right)
$$
(by definition of entropy) and so
\begin{equation}
\label{threshold} H(T_{{\mathcal C}(v)} | R_{w,A}) \leq \sum_{\sumrange} H\left(\Pr(\{T_{{\mathcal C}(v)}=\sumrange\} | R_{w,A})\right).
\end{equation}
For $A \in {\mathcal S}_1(H)$, we cannot do any better than bounding all $2^{|V(H)|m}$ entropy terms in (\ref{threshold}) by $1$, leading to
\begin{eqnarray}
\nonumber \sum_{A \in {\mathcal S}_1(H)} r_{w,A} H(T_{{\mathcal C}(v)} | R_{w,A}) & \leq & 2^{|V(H)|m} \sum_{A \in {\mathcal S}_1(H)} r_{w,A} \\
\label{last-en-term-2} & \leq & 2(3m-1)2^{|V(H)|(m+1)}\varepsilon,
\end{eqnarray}
since there are at most $2^{|V(H)|}$ summands and each is at most $2(3m-1)\varepsilon$. For $A \in {\mathcal S}_2(H)$, on the other hand, we know by Lemma \ref{Hlemma} and the definition of ${\mathcal S}_2(H)$ that
$$
\Pr(\{T_{{\mathcal C}(v)}=\sumrange\} | R_{w,A}) \leq \frac{(3m-1)\varepsilon}{r_{w,A}}  \leq  \frac{1}{2}
$$
if $\sumrange \neq \altAB$, while
$$
\Pr(\{T_{{\mathcal C}(v)}=\sumrange\} | R_{w,A}) \geq 1-\frac{(3m-1)\varepsilon}{r_{w,A}} \geq \frac{1}{2}
$$
if $\sumrange = \altAB$. We may therefore replace each of the entropy terms in (\ref{threshold}) by $H((3m-1)\varepsilon/r_{w,A})$, leading to
\begin{eqnarray}
\nonumber & & \sum_{A \in {\mathcal S}_2(H)} r_{w,A} H(T_{{\mathcal C}(v)} | R_{w,A}) \\
\nonumber & \leq & 2^{|V(H)|m} \sum_{A \in {\mathcal S}_2(H)} r_{w,A} H\left(\frac{(3m-1)\varepsilon}{r_{w,A}}\right) \\
\label{Jensen} & \leq & 2^{|V(H)|m} \left(\sum_{A \in {\mathcal S}_2(H)} r_{w,A}\right) H\left(\frac{|{\mathcal S}_2(H)|(3m-1)\varepsilon}{\sum_{A \in {\mathcal S}_2(H)} r_{w,A}}\right)
\end{eqnarray}
with (\ref{Jensen}) an application of Jensen's inequality. Now we use the fact that $\varepsilon \leq c/d$ to conclude that the argument of the entropy term in (\ref{Jensen}) is bounded above by $C\varepsilon$ for some constant depending on $m$ and $H$ (this utilizes Lemma \ref{Hlemma} and the fact that $\sum_{A \in {\mathcal S}_1(H)} r_{w,A}$ is at most $c\varepsilon$) to get
\begin{equation}
\label{last-en-term-3} \sum_{A \in {\mathcal S}_2(H)} r_{w,A} H(T_{{\mathcal C}(v)} | R_{w,A}) \leq CH(C\varepsilon).
\end{equation}

We now combine (\ref{last-en-term-1}), (\ref{last-en-term-2}) and (\ref{last-en-term-3}) with (\ref{KeyEquation}) to find that there are constants $c_i, i=1,\ldots,4$ (all depending on both $m$ and $H$) such that
$$
c_1\varepsilon \leq \frac{d^{c_2}}{m^d} + \frac{c_3H(c_4\varepsilon)}{d}.
$$
Using $H(x) \leq 2x\log(1/x)$ for $x \leq 1/2$ (a simple power series argument) this becomes
\begin{equation} \label{epsilon-second-shot}
c_1\varepsilon \leq \frac{d^{c_2}}{m^d} + \frac{c_3 \varepsilon}{d} \log \frac{1}{c_4\varepsilon},
\end{equation}
from which it follows that $\varepsilon \leq 2^{-\Omega(d)}$.

\section{Coloring with $q=q(d)$ colors} \label{sec-varyingwithd}

In the uniform proper $q$-coloring model ($H=K_q$, $\Lambda=(1,\ldots,1)$) it is natural to allow $q$, the number of colors, to vary with $d$ (see e.g. \cite{BrightwellWinkler3}, \cite{Jerrum}, \cite{Jonasson}, \cite{SalasSokal}). We may define long-range influence in this case exactly as in (\ref{def-lri}), simply allowing $H$ to also change with $d$.

The Dobrushin uniqueness theorem \cite{Dobrushin2} implies that we do not have long-range influence in the $q$-coloring model on $\torusMd$ when $q > 2d$ (in the case $m=2$) or $q>4d$ (in the case $m \geq 4$). On the other hand, Corollary \ref{cor-q-col-inf} establishes that we do have long-range influence for all constant $q$.

We can say a little bit more. Going through the proof of Theorem \ref{thm-small-prob-edge-not-ideal}, keeping careful track of the dependency of the final constants on $|V(H)|$, we find that we can prove the following theorem.
\begin{thm} \label{thm-q-growing-with-d}
Fix $m \geq 2$ even. If $f$ is chosen uniformly from ${\rm Hom}(\torusMd,K_q)$, for any $q < (\log d)/(m+2)$, then
$$
\Pr(\mbox{$e$ is not ideal with respect to $f$}) \leq d^{-4}.
$$
\end{thm}
(We could replace $(\log d)/(m+2)$ here with $c\log d$ for any $c<1/(m+1)$. We could also replace $d^{-4}$ by $d^{-C}$ for arbitrary $C>0$, but $d^{-4}$ is more than enough for our intended application.)
The proof of Theorem \ref{thm-q-growing-with-d} is straightforward, and we just mention some issues here. From (\ref{epsilon-first-shot}) we can no longer conclude that $\varepsilon \leq c/d$, but we do obtain $\varepsilon \leq \log^{O(m)} d/d$. In order to conclude that the argument in the entropy term in (\ref{Jensen}) is going to zero with $d$, we need only $q < c\log d$ for any $c<1$. In the final analysis we replace (\ref{epsilon-second-shot}) by
$$
\frac{c_1\varepsilon}{q^{2m}} \leq \frac{c_2d^{2m}\log q}{m^d} + \frac{c_32^{q(m+1)} \varepsilon}{d} \log \frac{1}{c_42^q \varepsilon}
$$
(with the numbered constants depending on $m$), from which the result follows.

Repeating the proofs of Theorems \ref{thm-structure-of-hom-space-coarse} and \ref{thm-structure-of-hom-space}, replacing appeals to Theorem \ref{thm-small-prob-edge-not-ideal} with appeals to Theorem \ref{thm-q-growing-with-d}, we then easily obtain the analog of Theorem \ref{thm-structure-of-hom-space} for the proper $q$-coloring model with $q < (\log d)/(m+2)$, with all occurrences of $2^{-\Omega(d)}$ in Theorem \ref{thm-structure-of-hom-space} replaced by $1/d$. This is more than enough to obtain the following long range influence result, using the scheme described in Section \ref{sec-influence}.
\begin{thm} \label{thm-long-range-col-q-large}
Fix $m \geq 2$ even. If $f$ is chosen uniformly from ${\rm Hom}(\torusMd,K_q)$ with $q < (\log d)/(m+2)$, then for any $x, y \in {\mathcal E}$ and $k \in \{1, \ldots, q\}$ we have
$$
\lim_{d \rightarrow \infty} \frac{\Pr(f(x)=k)}{\Pr(f(x)=k|f(y)=k)} = \frac{1}{2}.
$$
\end{thm}

\begin{conj}
If $q \leq d$ (in the case $m=2$) or $q \leq 2d$ (in the case $m\geq 4$) then there is long-range influence in the $q$-coloring model on $\torusMd$. Otherwise, there is no long-range influence.
\end{conj}

A motivation for this conjecture comes from the infinite $\Delta$-regular tree ${\mathbb T}_\Delta$. Let $f$ be a $q$-coloring of ${\mathbb T}_\Delta$. For each $\ell \geq 1$, let ${\vec p}^f_\ell$ be the occupation probability vector of a fixed vertex in a uniformly chosen $q$-coloring of ${\mathbb T}_\Delta$ conditioned on the coloring agreeing with $f$ on all vertices at graph distance more than $\ell$ from $x$. Brightwell and Winkler \cite{BrightwellWinkler3} showed that for $q \leq \Delta$, there are choices of $f$ for which ${\vec p}^f_\ell$ does not, in the limit as $\ell$ goes to infinity, approach the uniform vector. On the other hand Jonasson \cite{Jonasson} showed that for $q \geq \Delta + 1$ the limit is uniform for all $f$. In other words, $q = \Delta$ is the threshold for long-range influence, suitably interpreted, in ${\mathbb T}_\Delta$.

\section{Discussion and open problems}\label{sec-openprobs}

\subsection{The sizes of the partition classes}

Theorem \ref{thm-structure-of-hom-space} does not give any information about the relative ($\Lambda$-weighted) sizes of the $C_\Lambda(A,B)$'s. We give two examples here to show that many different behaviors are possible, making such a general statement rather difficult to formulate.

A fact that we use in both examples is that for $G$ connected and $H$ consisting of components $H_1$ and $H_2$ we can identify ${\rm Hom}(G,H)$ with the disjoint union of ${\rm Hom}(G,H_1)$ and ${\rm Hom}(G,H_2)$.

First, consider $H$ the disjoint union of $H_{\rm ind}$ and $K_3$ (note that $\eta(H_{\rm ind})=\eta(K_3)=2$) with $\Lambda = (1, \ldots, 1)$. The results of \cite{KorshunovSapozhenko} and \cite{Galvin-HomstoZ} (see (\ref{KS-result}), (\ref{G-count-of-3-cols}) and the discussions around these equations) together imply that in any decomposition of ${\rm Hom}(\torustwod,H)$ satisfying the conditions of Theorem \ref{thm-structure-of-hom-space}, along with the exceptional class we have eight partition classes. Six of these correspond to the six elements of ${\mathcal M}(K_3)$, and these each have size $(1+o(1))e/(6e+2\sqrt{e})|{\rm Hom}(\torustwod,H)| \approx .14 |{\rm Hom}(\torustwod,H)|$. The two remaining classes correspond to the two elements of ${\mathcal M}(H_{\rm ind})$ and each have size $(1+o(1))\sqrt{e}/(6e+2\sqrt{e})|{\rm Hom}(Q_d,H)| \approx .08 |{\rm Hom}(\torustwod,H)|$.

For an example with a different type of behavior, let $H$ be the disjoint union of $K_4^{{\rm loop}}$ (the complete looped graph on four vertices) and $K_8$ (note that $\eta(K_8)=\eta(K_4^{{\rm loop}})=16$, with ${\mathcal M}(K_4^{{\rm loop}})=(V(K_4^{{\rm loop}}),V(K_4^{{\rm loop}}))$), again with $\Lambda = (1, \ldots, 1)$. It is immediate that $|{\rm Hom}(\torustwod,K_4^{{\rm loop}})|=16^{2^{d-1}}$ and that all colorings in this set are pure-$(V(K_4^{{\rm loop}}),V(K_4^{{\rm loop}}))$ colorings. It is also fairly straightforward to verify that $|{\rm Hom}(\torustwod,K_8)|=\omega(16^{2^{d-1}})$. Indeed, consider proper $8$-colorings of $\torustwod$ which are pure-$(A,B)$ for some $(A,B)$, except that there is one vertex from ${\mathcal E}$ that is colored from $B$. An easy count gives that there are $(1/2)(3/2)^d 16^{2^{d-1}}$ such colorings. This implies that in any decomposition of ${\rm Hom}(\torustwod,H)$ satisfying the conditions of Theorem \ref{thm-structure-of-hom-space}, along with the exceptional set we have ${8 \choose 4}+1$ partition classes. The first ${8 \choose 4}$ of these classes correspond to the elements of ${\mathcal M}(K_8)$ and each have size $\Omega(|{\rm Hom}(Q_d,H)|)$, and the last class corresponds to the unique element of ${\mathcal M}(K_4^{{\rm loop}})$ and has size $o(|{\rm Hom}(Q_d,H)|)$.

There is a fairly natural conjecture concerning the sizes of the $C_\Lambda(A,B)$'s in general, which we now discuss. A trivial lower bound on $w_\Lambda(C_\Lambda(A,B))$ is
$$
w_\Lambda(C_\Lambda(A,B)) \geq (\eta_\Lambda(H))^{m^d/2}.
$$
A better lower bound is obtained by the following process. First, for each $s, t \in {\mathbb N}$ with $s, t \leq U$ (some appropriately chosen upper bound), select $S \subseteq {\mathcal E}$ and $T \subseteq {\mathcal O}$ with $|S|=s$ and $|T|=t$ and with the property that for each $x, y \in S \cup T$, we have $x \cup N_x$ disjoint from $y \cup N_y$. For $U$ not too large, the number of choices for $(S,T)$ is close to $((m^d)^{s+t})/(2^{s+t}s!t!)$. Next, choose a color from $A$ for each $v \in {\mathcal E} \setminus (S \cup N(T))$, a color from $B$ for each $w \in {\mathcal O} \setminus (T \cup N(S))$, a color from $A^c$ for each $v \in S$ and a color from $B^c$ for each $w \in T$. Finally, for each $v \in S$ (resp. $w \in T$) choose a color for each vertex of $N_v$ (resp. $N_w$) from among those colors which are adjacent to everything in $A$ (resp. $B$) as well as to the color chosen for $v$ (resp. $w$). For each $k \not \in A$, let $N(A,k)$ be the set of colors adjacent to everything in $A$ as well as to $k$, and for $\ell \not \in B$ define $N(B,\ell)$ analogously.

For each choice of $S$ and $T$ with $|S|=s$ and $|T|=t$, the sum of the weights of all the colorings obtained by the process described above is
$$
\lambda_A^{m^d/2 - s - \Delta t} \lambda_B^{m^d/2 - t - \Delta s} \left(\sum_{k \not \in A} \lambda_k \lambda^\Delta_{N(A,k)}\right)^s \left(\sum_{\ell \not \in B} \lambda_\ell \lambda^\Delta_{N(B,\ell)}\right)^t.
$$
(To avoid having to separate the cases $m=2$ and $m\geq 4$ we use $\Delta$ to denote the degree of a vertex in $\torusMd$.) Summing over all $s$ and $t$, as long as $U$ is large enough we get a lower bound of the form
$$
w_\Lambda(C_\Lambda(A,B)) \geq \eta_\Lambda(H)^{\frac{m^d}{2}}\exp\left\{m^dL_\Lambda(A,B,d)(1+o(1))\right\}
$$
where
$$
L_\Lambda(A,B,d) = \frac{1}{2\lambda_A\lambda^\Delta_B} \sum_{k \not \in A} \lambda_k \lambda^\Delta_{N(A,k)}        + \frac{1}{2\lambda_B\lambda^\Delta_A} \sum_{\ell \not \in B} \lambda_\ell \lambda^\Delta_{N(B,\ell)}.
$$
We conjecture that this lower bound is essentially the truth.
\begin{conj} \label{conj-count}
For all $H$, $\Lambda$ and $m \geq 2$ even, there is a decomposition of ${\rm Hom}(\torusMd,H)$ satisfying the conditions of Theorem \ref{thm-structure-of-hom-space} and moreover satisfying that for each $(A,B) \in {\mathcal M}_\Lambda(H)$ we have
$$
w_\Lambda(C_\Lambda(A,B)) = \eta_\Lambda(H)^{\frac{m^d}{2}}\exp\left\{m^dL_\Lambda(A,B,d)(1+o(1))\right\}
$$
as $d \rightarrow \infty$.
\end{conj}
This conjecture is true in the case $H=H_{\rm ind}$, $m=2$ (that is, $G=\torustwod$) and $\Lambda=(1,\lambda)$ (unlooped vertex listed first) for all $\lambda > 0$ (for $\lambda=1$ this is implicit in the work of Korshunov and Sapozhenko \cite{KorshunovSapozhenko}, and for all other $\lambda$ it is implicit in the work of Galvin \cite{Galvin-Qdthresh}). It is also true in the case $H=K_3$, $m=2$ and $\Lambda=(1, 1, 1)$ (this is implicit in the work of Galvin \cite{Galvin-3colQdmix}).

An appealing special case of Conjecture \ref{conj-count} is a count of the set ${\mathcal C}_q(\torustwod)$ of proper $q$-colorings of $\torustwod$ ($H=K_q$, $\Lambda=(1, \ldots, 1)$).
\begin{conj} \label{conj-count-col}
For all $q \in {\mathbb N}$,
$$
|{\mathcal C}_q(\torustwod)| =
(1+{\mathbf 1}_{\{q~{\rm odd}\}}){q \choose \lfloor q/2\rfloor}\left(\lfloor q/2 \rfloor \lceil q/2
\rceil \right)^{2^{d-1}} \exp\left\{f(q)(1+o(1)) \right\}
$$
as $d \rightarrow \infty$, where
$$
f(q) = \frac{\lceil q/2
\rceil}{2\lfloor q/2 \rfloor}\left(2- \frac{2}{\lceil q/2 \rceil}
\right)^d + \frac{\lfloor q/2 \rfloor}{2\lceil q/2 \rceil}\left(2-
\frac{2}{\lfloor q/2 \rfloor} \right)^d.
$$
\end{conj}
This is proved for $q=3$ in \cite{Galvin-HomstoZ}.

\subsection{Mixing time and the size of the exceptional class}

The ($\Lambda$-weighted) size of the exceptional class $C_\Lambda(0)$ of Theorem \ref{thm-structure-of-hom-space} is closely related to the mixing time of local-update algorithms designed to sample from ${\rm Hom}(\torusMd,H)$ according to the distribution $p_\Lambda$.

Fix $H$, $\Lambda$ and $m$. Let ${\mathcal W}$ be an ergodic, time homogeneous Markov chain on state space ${\rm Hom}(\torusMd,H)$ with transition probabilities $P(f,g)$ for $f, g \in {\rm Hom}(\torusMd,H)$ and stationary distribution $p_\Lambda$. Assume that ${\mathcal W}$ is {\em local}; for the purposes of this section, that means that there is a function $\rho(d)=o(m^d)$ such that if $f$ and $g$ differ at more than $\rho(d)$ vertices then $P(f,g)=0$.

An example of such a chain is {\em Glauber dynamics}, which makes transitions from $f$ as follows: first choose a vertex $v$ of $\torusMd$ uniformly, then choose a coloring to transition to from among the set of colorings which agree with $f$ off $v$, with each such coloring $g$ being chosen with probability proportional to $\lambda_{g(v)}$.

The {\em mixing time} $\tau_{\rm mix}({\mathcal W})$ of such a chain is defined to be the smallest time $t$ such that after running the chain for $t$ steps, from an arbitrary starting state, it is certain that the distribution of the chain is within $1/e$ (say; any constant less than $1/2$ will do) of $p_\Lambda$ in total variation distance. This captures how effective the chain is at generating a sample that is guaranteed to be within any prescribed distance of the stationary distribution; in particular, if one wishes for a sample that is from a distribution within $(1/e)^c$ of stationary, it is sufficient to run the chain for $c\tau_{\rm mix}({\mathcal W})$ steps. The chain ${\mathcal W}$ is said to mix {\em rapidly} if $\tau_{\rm mix}({\mathcal W})$ is a polynomial in $m^d$, and {\em slowly} otherwise. (See e.g. \cite{LevinPeresWilmer} for a thorough treatment.)

Let $(A,B) \in {\mathcal M}_\Lambda(H)$ be such that $w_\Lambda(C_\Lambda(A,B))/Z_\Lambda(\torusMd,H)$ is bounded away from $0$ and is at most $1/2$ (this will happen, for example, if $|{\mathcal M}_\Lambda(H)|\geq 2$ and the partition of ${\rm Hom}(\torusMd,H)$ guaranteed by Theorem \ref{thm-structure-of-hom-space} is an approximate equipartition). By the properties of the partition and the locality of ${\mathcal W}$, it is clear that in any step in which the chain leaves $C_\Lambda(A,B)$, it must go to $C_\Lambda(0)$. This suggests that the mixing time of the chain might be high, since $C_\Lambda(0)$ acts as a bottleneck.

This intuition may be formalized using the notation of the conductance of a chain, introduced by Jerrum and Sinclair \cite{JerrumSinclair2}. Using the form of the conductance argument presented in \cite{DyerFriezeJerrum} (see \cite{Galvin-3colQdmix}, \cite{GalvinTetali-indQdmix} for specific applications in a setting similar to the present one), it follows that
\begin{equation} \label{mixing-time}
\tau_{\rm mix}({\mathcal W}) \geq \frac{w_\Lambda(C_\Lambda(A,B))}{8w_\Lambda(C_\Lambda(0))} \geq \Omega\left(\frac{Z_\Lambda(\torusMd,H)}{w_\Lambda(C_\Lambda(0))}\right).
\end{equation}

In the presence of Theorem \ref{thm-structure-of-hom-space}, the lower bound on $\tau_{\rm mix}({\mathcal W})$ given by (\ref{mixing-time}) is $2^{O(d)}$, which conveys no information since this is only polynomial in $m^d$. We believe, however, that is should be possible to find a much smaller upper bound on $C_\Lambda(0)$ that would in particular give an exponential lower bound on $\tau_{\rm mix}({\mathcal W})$.
\begin{conj} \label{conj-small-exp}
Fix $H$, $\Lambda$ and $m \geq 2$ even. There is a partition of ${\rm Hom}(\torusMd,H)$ satisfying all the conditions of Theorem \ref{thm-structure-of-hom-space} as well as
$$
w_\Lambda(C(0)) = 2^{-g(d)m^d} Z_\Lambda(\torusMd,H)
$$
for some polynomial $g(d)$ (whose degree depends only on $H$, $\Lambda$ and $m$).
\end{conj}
One way to prove this conjecture would be to obtain a concentration result showing that for $f$ chosen from ${\rm Hom}(\torusMd,H)$ according to $p_\Lambda$, with high probability the number of non-ideal edges is close to its expected value; we are currently using the very weak Markov's inequality.

The slow mixing result that would be implied by Conjecture \ref{conj-small-exp} has been obtained for various special cases (\cite{BorgsChayesDyerTetali} for a large class of $H$ with carefully chosen $\Lambda$, \cite{Galvin-indZdmix} and \cite{GalvinTetali-indQdmix} for $H=H_{ind}$ and $\Lambda=(1,\lambda)$ for all fixed $\lambda > 0$, and \cite{Galvin-3colQdmix} and \cite{GalvinRandall} for $H=K_3$ and $\Lambda=(1, \ldots, 1)$).

\subsection{Varying $m$ with $d$}

All of our results are for fixed $m$, and become interesting as $d$ grows. It would be of great interest to obtain similar results for fixed $d$, as $m$ grows (as Peled \cite{Peled} has done in the case $H=K_3$), as this would allow us to say something about the space of Gibbs measures for the probability distribution $p_\Lambda$ on the infinite space ${\rm Hom}({\mathbb Z}^d, H)$ (see for example \cite{BorgsChayesDyerTetali}, \cite{BrightwellWinkler}, for a discussion of Gibbs measures in the specific context of homomorphism models). Unfortunately, a careful examination of our proof of Theorem \ref{thm-small-prob-edge-not-ideal}, keeping track of the dependency of the final constants on $m$, shows that at best we may take $m = c\log d$ for some absolute constant $c>0$ if we wish to obtain useful results.


\begin{thebibliography}{99}

\bibitem{AlonSpencer} N. Alon and J. Spencer, {\em The
Probabilistic Method}, Wiley, New York, 2000.

\bibitem{BenjaminiHaggstromMossel}
I. Benjamini, O. H\"aggstr\"om and E. Mossel, On random graph
homomorphisms into ${\mathbb Z}$, {\em J. Combin. Theory Ser. B}
{\bf 78} (2000), 86--114.

\bibitem{Bollobas2}
B. Bollob\'as,
{\em Modern Graph Theory}, Springer, New York, 1998.

\bibitem{BollobasLeader2}
B. Bollob\'as and I. Leader, Edge-isoperimetric
inequalities in the grid, {\em Combinatorica} {\bf 11} (1991),
299--314.

\bibitem{BorgsChayesDyerTetali}
C. Borgs, J. Chayes, M. Dyer and P. Tetali, On the sampling problem for $H$-colorings on the hypercubic lattice, DIMACS Series in Discrete Mathematics and Theoretical Computer Science {\bf 63} (2004) {\em
Graphs, Morphisms and Statistical Physics}, 13--28.

\bibitem{BorgsChayesFriezeKimTetaliVigodaVu}
C. Borgs, J. Chayes, A. Frieze, J. Kim, P. Tetali, E. Vigoda and V. Vu, Torpid Mixing of Some Monte Carlo Chain
Algorithms in Statistical Physics, {\em Proc. IEEE FOCS} (1999), 218--229.

\bibitem{BrightwellWinkler}
G. Brightwell and P. Winkler, Graph homomorphisms and phase
transitions, {\em J. Combin. Theory Ser. B} {\bf 77} (1999),
221--262.

\bibitem{BrightwellWinkler2}
G. Brightwell and P. Winkler,
Hard constraints and the Bethe lattice: adventures at the interface of combinatorics and statistical physics, {\em Proc. Int'l. Congress of Mathematicians Vol. III} (Li Tatsien, ed.), Higher Education Press, Beijing (2002), 605--624.

\bibitem{BrightwellWinkler3}
G. Brightwell and P. Winkler,
Random colorings of a Cayley tree, Bolyai Society Mathematical Studies {\bf 1} (2002) {\em Contemporary Combinatorics}, 247--276.

\bibitem{ChungFranklGrahamShearer}
F. Chung, P. Frankl, R. Graham
and J. Shearer, Some intersection theorems for ordered sets and
graphs, {\em J. Combin. Theory Ser. A} {\bf 48} (1986), 23--37.

\bibitem{Diestel}
R. Diestel, {\em Graph Theory}, Springer, Heidelberg, 2005.

\bibitem{Dobrushin2}
R. Dobrushin, Prescribing a system of random variables
by the help of conditional distributions, {\em Theory Probab. Appl.} {\bf 15} (1970), 469-�497.

\bibitem{DyerFriezeJerrum}
M. Dyer, A. Frieze and M. Jerrum, On counting independent sets in
sparse graphs, {\em SIAM J. Comput.} {\bf 31} (2002), 1527--1541.

\bibitem{EngbersGalvin}
J. Engbers and D. Galvin, $H$-colouring bipartite graphs, {\em J. Combin. Theory Ser. B} {\bf 102} (2012), 726-–742

\bibitem{Galvin-Qdthresh}
D. Galvin, A threshold phenomenon for random independent sets in the discrete hypercube, {\em Combin. Probab. Comput.} {\bf 20} (2011) 27--51.

\bibitem{Galvin-HomstoZ}
D. Galvin, On homomorphisms from the Hamming cube to ${\mathbb Z}$, {\em Israel J. Math.} {\bf 138} (2003), 189--213.

\bibitem{Galvin-indZdmix}
D. Galvin, Sampling independent sets on the discrete torus, {\em Random Structures Algorithms} {\bf 33} (2008), 356--376.

\bibitem{Galvin-3colQdmix}
D. Galvin, Sampling $3$-colourings of regular bipartite graphs, {\em Electron. J. Probab.} {\bf 12} (2007), 481--497.

\bibitem{Galvin-spin}
D. Galvin, Bounding the partition function of spin systems, {\em Electron. J. Combin.} {\bf 13} (2006), R72.

\bibitem{GalvinKahn}
D. Galvin and J. Kahn, On phase transition in the hard-core model on ${\mathbb Z}^d$, {\em Combin. Probab. Comput.} {\bf 13} (2004), 137-164.

\bibitem{GalvinRandall}
D. Galvin and D. Randall, Torpid Mixing of Local Markov Chains on 3-Colorings of the Discrete Torus, {\em Proc. ACM--SIAM SODA} (2007), 376--384.

\bibitem{GalvinTetali-weighted}
D. Galvin and P. Tetali, On weighted graph homomorphisms, DIMACS Series in Discrete Mathematics and Theoretical Computer Science {\bf 63} (2004) {\em
Graphs, Morphisms and Statistical Physics},
97--104.

\bibitem{GalvinTetali-indQdmix}
D. Galvin and P. Tetali, Slow mixing of Glauber dynamics for the hard-core model on regular bipartite graphs, {\em Random Structures Algorithms} {\bf 28} (2006), 427-443.

\bibitem{Hoeffding}
W. Hoeffding, Probability inequalities for sums of bounded random variables,
{\em J. American Statistical Association} {\bf 58} (1963), 13�-30.

\bibitem{Jerrum}
M. Jerrum, A very simple algorithm for estimating the number of
$k$-colourings of a low-degree graph, {\em Random Structures Algorithms}
{\bf 7} (1995), 157--165.

\bibitem{JerrumSinclair2}
M. Jerrum and A. Sinclair, Conductance and the rapid mixing property
for Markov chains: the approximation of the permanent resolved, {\em
Proc. ACM STOC '88}, 235--243.

\bibitem{Jonasson}
J. Jonasson, Uniqueness of uniform random colorings of regular
trees, {\em Statist. Probab. Lett.}  {\bf 57} (2002), 243-248.

\bibitem{Kahn}
J. Kahn, An Entropy Approach to the Hard-Core Model on Bipartite Graphs,
{\em Combin. Probab. Comput.} {\bf 10} (2001),
219--237.

\bibitem{Kahn2}
J. Kahn, Range of cube-indexed random walk, {\em Israel J.
Math.} {\bf 124} (2001), 189--201.

\bibitem{KorshunovSapozhenko} A. Korshunov and A. Sapozhenko, The number of
binary codes with distance $2$, {\em Problemy Kibernet.} {\bf 40}
(1983), 111--130. (Russian)

\bibitem{LevinPeresWilmer}
D. Levin, Y. Peres and E. Wilmer, {\em Markov Chains and Mixing Times}, AMS, Providence, 2009.

\bibitem{McEliece}
R.J. McEliece,
{\em The Theory of Information and Coding}, Addison-Wesley,
London, 1977.

\bibitem{Peled}
R. Peled, High-Dimensional Lipschitz Functions are Typically Flat, arXiv:1005.4636

\bibitem{WidomRowlinson}
J. Rowlinson and B. Widom, New Model for the Study of Liquid-Vapor Phase Transitions, {\em J. Chem. Phys.} {\bf 52} (1970), 1670--1684.

\bibitem{SalasSokal}
J. Salas and A. Sokal, Absence of phase transition for antiferromagnetic Potts
models via the Dobrushin uniqueness theorem, {\em J. Stat. Physics} {\bf 86} (1997), 551--579.

\bibitem{Sokal}
A. Sokal, A Personal List of Unsolved Problems Concerning Lattice
Gases and Antiferromagnetic Potts Models, {\em Markov Processes
and Related Fields} {\bf 7} (2001), 21--38.

\end{thebibliography}
\end{document}